\documentclass{amsart}

\usepackage[a4paper,hmargin=2.9cm,vmargin=4cm]{geometry}
\usepackage{amsfonts,amssymb,amscd,amstext}
\usepackage{graphicx}
\usepackage[dvips]{epsfig}
\usepackage{hyperref}

\renewcommand{\leq}{\leqslant}
\renewcommand{\geq}{\geqslant}
\newcommand{\ptl}{\partial}
\newcommand{\rr}{{\mathbb{R}}}
\newcommand{\la}{\lambda}
\newcommand{\hh}{{\mathbb{H}}}
\newcommand{\h}{\mathcal{H}}

\newcommand{\sub}{\subset}
\newcommand{\subeq}{\subseteq}
\newcommand{\escpr}[1]{\big<#1\big>}
\newcommand{\Sg}{\Sigma} 
\newcommand{\Om}{\Omega}
\newcommand{\om}{\omega}
\newcommand{\eps}{\varepsilon}
\newcommand{\var}{\varphi}
\newcommand{\ga}{\gamma}

\newcommand{\mnh}{|N_{h}|}
\newcommand{\nuh}{\nu_{h}}
\newcommand{\ric}{\text{Ric}}

\newcommand{\m}{\mathbb{N}(\kappa)}
\newcommand{\e}{\mathbb{M}(\kappa)}
\newcommand{\stres}{M}

\DeclareMathOperator{\divv}{div}

\newtheorem{theorem}{Theorem}[section]
\newtheorem{proposition}[theorem]{Proposition}
\newtheorem{lemma}[theorem]{Lemma}
\newtheorem{corollary}[theorem]{Corollary}

\theoremstyle{definition}

\newtheorem{remark}[theorem]{Remark}

\theoremstyle{remark}

\numberwithin{equation}{section}

\begin{document}

\title[complete stable cmc surfaces in Sasakian sub-Riemannian $3$-manifolds]
{complete stable cmc surfaces with empty singular set \\ in Sasakian sub-Riemannian $3$-manifolds}

\author[C.~Rosales]{C\'esar Rosales}
\address{Departamento de Geometr\'{\i}a y Topolog\'{\i}a \\
Universidad de Granada \\ E--18071 Granada \\ Spain}
\email{crosales@ugr.es}

\date{July 15, 2010}

\thanks{The author has been supported by MCyT-Feder grant
MTM2010-21206-C02-01 and J.~A. grant P09-FQM-5088} 

\begin{abstract}
For constant mean curvature surfaces of class $C^2$ immersed inside Sasakian sub-Riemannian $3$-manifolds we obtain a formula for the second derivative of the area which involves horizontal analytical terms, the Webster scalar curvature of the ambient manifold, and the extrinsic shape of the surface. Then we prove classification results for complete surfaces with empty singular set which are stable, i.e., second order minima of the area under a volume constraint, inside the $3$-dimensional sub-Riemannian space forms. In the first Heisenberg group we show that such a surface is a vertical plane. In the sub-Riemannian hyperbolic $3$-space we give an upper bound for the mean curvature of such surfaces, and we characterize the horocylinders as the only ones with squared mean curvature $1$. Finally we deduce that any complete surface with empty singular set in the sub-Riemannian $3$-sphere is unstable. 
\end{abstract}

\maketitle

\thispagestyle{empty}

\section{Introduction}
\label{sec:intro}
\setcounter{equation}{0} 

The first variation formulae for area and volume in Riemannian $3$-manifolds imply that constant mean curvature surfaces (CMC surfaces) are critical points of the area under a volume constraint. They provide the first order candidates to solve the \emph{isoperimetric problem}, which seeks sets of least boundary area among those enclosing a given volume. In general, not any CMC surface bounds an isoperimetric region. Hence it is natural to compute the second variation and to consider second order minima of the area under a volume constraint. CMC surfaces with non-negative second derivative of the area for compactly supported volume-preserving variations are called \emph{stable surfaces} and have been extensively investigated in the literature.

Several authors have studied the stability condition of CMC surfaces in Riemannian $3$-space forms. The papers by Barbosa and do Carmo \cite{bdc}, and by Barbosa, do Carmo and Eschenburg \cite{bdce}, show that a compact, orientable, stable CMC surface immersed inside a Riemannian $3$-space form must be a geodesic sphere. The complete non-compact case was analyzed later. Da Silveira in \cite{silveira}, and L\'opez and Ros in \cite{lopez-ros}, used results by Fischer-Colbrie \cite{fc} to prove that, in Euclidean $3$-space, a complete, non-compact, orientable, stable CMC surface is a plane. Also as a consequence of \cite{fc} it follows that a complete stable surface inside the $3$-sphere must be compact, and so it coincides with a geodesic sphere by \cite{bdce}. In the hyperbolic $3$-space, Da Silveira \cite{silveira} characterized the horospheres as the unique complete, non-compact, stable CMC surfaces whose mean curvature $H$ satisfies $H^2\geq 1$. He also described in \cite{silveira} a one-parameter family of stable, non-umbilic, embedded CMC surfaces with $0<H^2<1$. Examples of non-totally geodesic area-minimizing surfaces in the hyperbolic $3$-space were previously discovered in \cite{hass}. 

Motivated by classical area-minimizing problems, the study of variational questions related to the area in \emph{sub-Riemannian geometry} has been recently focus of attention. In particular, the theory of minimal and CMC surfaces in sub-Riemannian $3$-manifolds has received an increasing development in the last years. In this paper we discuss some aspects of this theory in the setting of Sasakian $3$-manifolds, with special emphasis on the sub-Riemannian $3$-space forms.

There are several ways of introducing \emph{Sasakian sub-Riemannian $3$-manifolds}. They are contact sub-Riemannian $3$-manifolds for which the diffeomorphisms associated to the Reeb vector field are isometries (see Section~\ref{subsec:ssR3m} for a precise definition). This is equivalent to that the Webster torsion of the underlying pseudo-Hermitian $3$-manifold \cite{webster}, \cite{chern} vanishes identically. From another point of view, a Sasakian sub-Riemannian manifold is obtained from a Sasakian Riemannian manifold by restriction of the Sasaki metric to the contact plane. We should stress that Sasakian geometry has many connections with different branches of Mathematics and Physics. We refer the reader to \cite{boyer} for an excellent introduction to the subject which includes a wide variety of examples. 

In the study of the equivalence problem in sub-Riemannian geometry it is shown, by means of the Tanaka connection \cite{webster}, \cite[Sect.~10.4]{blair}, \cite{falbel3}, that a simply connected, homogeneous, contact sub-Riemannian $3$-manifold can be described in terms of some geometric and algebraic invariants, see \cite{falbel1}. For a Sasakian sub-Riemannian $3$-manifold the most important invariant is the \emph{Webster scalar curvature} $K$. We define it in \eqref{eq:webster} by following the approach given in \cite[Sect.~10.4]{blair} for arbitrary contact metric structures. However, it is not difficult to check that the function $K$ coincides, up to some constant, with the sectional curvature of the contact plane computed with respect to the Tanaka connection.  As a consequence, $K$ is invariant under sub-Riemannian isometries. Moreover, the function $K$ determines locally a Sasakian sub-Riemannian $3$-manifold. In particular, if two such manifolds have the same constant Webster scalar curvature then they are locally isometric \cite[Thm.~1.2]{falbel1}. All this might suggest that $K$ plays for a Sasakian $3$-manifold the same role as the Gaussian curvature for a Riemannian surface.

As in Riemannian geometry, it is natural to consider first those sub-Riemannian spaces which represent, in some sense, the most simple and symmetric examples of the theory. In dimension $3$ these are, respectively, the first Heisenberg group $\hh^1$, the special unitary group SU(2), and the universal cover of the special lineal group SL(2,\,$\rr$), see \cite{montgomery} and \cite{gromov-cc}. We introduce them by means of a unified approach as the Sasakian $3$-manifolds $\e$ in Section~\ref{subsec:sf}. There are some reasons for which the spaces $\e$ may be understood as sub-Riemannian counterparts to the Riemannian $2$-spaces forms. For example, it is known that any simply connected, homogeneous, contact sub-Riemannian $3$-manifold has an isometry group of dimension $3$ or $4$, see \cite{hughen}, \cite{diniz} and \cite{falbel1}. From this perspective, the spaces $\e$ are the unique ones with isometry group of maximal dimension. Moreover, by a result of Tanno \cite{tanno}, see also \cite[Thm.~1.2 and Rem.~1.1]{falbel1}, the space $\e$ is the unique complete, simply connected, Sasakian sub-Riemannian $3$-manifold with constant Webster scalar curvature $\kappa$. The manifolds $\e$ are also spaces of bundle type \cite{montgomery} since their sub-Riemannian structure comes from a submersion onto a Riemannian $2$-space form. Furthermore, the associated extended metric $g_\kappa$ defined in \eqref{eq:Rm} gives rise to one of the Thurston homogeneous geometries, see \cite{scott}. As a matter of fact, the Carnot-Carath\'eodory distance of $\e$ can be realized as the limit in the Hausdorff-Gromov topology of a family of Riemannian distances obtained by dilations of $g_\kappa$ in the vertical direction. The Heisenberg group $\mathbb{M}(0)$ is also an example of a Carnot group since its Lie algebra is nilpotent and stratified. This property does not hold for $\e$ if $\kappa\neq 0$. In particular, this shows that the family of Sasakian sub-Riemannian $3$-manifolds is not included in that of Carnot groups.

Once we have introduced in Sections~\ref{subsec:ssR3m} and \ref{subsec:sf} the sub-Riemannian structures in which we are interested, we turn back to the study of local minimizers of the area with or without a volume constraint.

There are several notions of \emph{volume} and \emph{area} for contact sub-Riemannian manifolds. In this paper we follow the point of view in \cite{rr1}, \cite{rr2} and \cite{hr1}: the volume of a Borel set is the Riemannian volume associated to the extended Riemannian metric \eqref{eq:Rm}, whereas the area of a surface is defined by means of the integral formula \eqref{eq:area}, which gives an expression for the \emph{horizontal perimeter \'a la De Giorgi} \eqref{eq:per} of a set with smooth boundary. 

In Section~\ref{sec:1ndvar} we gather some facts about $C^2$ \emph{area-stationary} surfaces with or without a volume constraint inside Sasakian sub-Riemannian $3$-manifolds. By using the arguments in \cite{rr2} we can deduce that, off of the \emph{singular set} given by the points where the tangent plane coincides with the contact one, the \emph{mean curvature} $H$ defined in \eqref{eq:mc} is constant, and the characteristic curves are ambient geodesics of curvature $H$ with respect to the Carnot-Carath\'eodory distance (CC-geodesics). These properties have been also studied in pseudo-Hermitian $3$-manifolds \cite{chmy}, vertically rigid manifolds \cite{hp1}, Carnot groups \cite{montefalcone}, and even in general sub-Riemannian manifolds \cite{hp2}. They suppose the starting point of the theory of CMC surfaces in sub-Riemannian geometry. One of the main goals of the theory is to establish classification theorems for complete CMC surfaces. The best results in this direction have been proved for surfaces with non-empty singular set. In fact, by using the description of the singular set of a CMC surface \cite{chmy}, and the orthogonality property between the characteristic curves and the curves contained in the singular set \cite{rr2}, it was possible to give a description of all the complete volume-preserving area-stationary surfaces with non-empty singular set in the first Heisenberg group $\mathbb{M}(0)$ and in the sub-Riemannian $3$-sphere $\mathbb{M}(1)$, see \cite{rr2} and \cite{hr1}. 

In the case of CMC surfaces with empty singular set the complete classification is far from being established even for the space forms $\e$. Let us summarize some related works. If $\Sg$ is a CMC surface with empty singular set in $\e$ then the \emph{horizontal Gauss map} defined in \eqref{eq:nuh} is continuous on $\Sg$, and the surface is ruled by CC-geodesics of the same curvature. It follows from the description given in Section~\ref{sec:ccgeo} that, if $\Sg$ is complete and $\kappa\leq 0$, then $\Sg$ cannot be compact. In the first Heisenberg group $\mathbb{M}(0)$ additional results have been obtained for minimal surfaces (those with vanishing mean curvature), see \cite{gp}, \cite{chenghwang}, \cite{bscv}, and for CMC surfaces of revolution \cite{rr1}. In the sub-Riemannian $3$-sphere $\mathbb{M}(1)$ it is known, as a consequence of  \cite[Thm.~E]{chmy}, that a compact CMC surface $\Sg$ with empty singular set is topologically a torus.  In \cite{hr1} it was proved that, if the mean curvature $H$ of $\Sg$ satisfies $H/\sqrt{1+H^2}\in\mathbb{Q}$, then $\Sg$ is congruent to a vertical Clifford torus. This result is not true if $H/\sqrt{1+H^2}\in\rr\setminus\mathbb{Q}$, as examples of unduloidal CMC tori were discovered in \cite{hr1}. Other important surfaces with empty singular set are the so-called \emph{vertical surfaces}: they satisfy that the Reeb vector field of the ambient manifold is always tangent. Complete CMC vertical surfaces have been classified in $\e$ for $\kappa\geq 0$, see \cite{gp}, \cite{rr2} and \cite{hr1}. In Proposition~\ref{prop:vert} of the paper we give a similar characterization of such surfaces for general Sasakian sub-Riemannian $3$-manifolds.

In spite of the aforementioned results, it seems that the CMC condition is not enough in order to prove optimal classification results for complete surfaces with empty singular set.
This provides a motivation for the study of second order minima of the area with or without a volume constraint. As in Riemannian geometry, we define in Section~\ref{sec:2ndvar} a \emph{stable surface under a volume constraint} as a $C^2$ volume-preserving area-stationary surface with non-negative second derivative of the area under compactly supported variations which leaves invariant the volume enclosed by the surface. Following the classical approach in the context of minimal surfaces, we say that an area-stationary surface is \emph{stable} if the second derivative of the area is nonnegative for variations \emph{which do not necessarily preserve volume}. It is then clear that, for an area-stationary surface, the condition of being stable is stronger than the one of being stable under a volume constraint.

The study of the stability in sub-Riemannian $3$-space forms has focused on area-stationary surfaces in the first Heisenberg group $\mathbb{M}(0)$. In \cite{bscv} it was proved that the Euclidean vertical planes are the unique entire stable area-stationary intrinsic graphs in $\mathbb{M}(0)$. The same characterization was obtained in \cite{dgnp} for entire area-stationary Euclidean graphs with empty singular set. These results suggest that any complete stable area-stationary surface with empty singular set in $\mathbb{M}(0)$ must be a vertical plane. This was definitely proved in \cite{dgnp-stable} for embedded surfaces, and in \cite{hrr} for immersed ones. Also in \cite{hrr}, complete stable area-stationary surfaces with non-empty singular set in $\mathbb{M}(0)$ were classified: they are either horizontal Euclidean planes, or congruent to the hyperbolic paraboloid $t=xy$. 

In this paper we use a unified approach inspired in \cite{hrr} to characterize complete stable surfaces with empty singular set in any space form $\e$. Our motivations to consider surfaces with empty singular set are twofold. On the one hand, it is natural to expect that the stability condition for such surfaces should imply more rigidness than the CMC condition. On the other hand, such surfaces are interesting from the analytical point of view since they possess a well-behaved horizontal Gauss map. In Theorem~\ref{th:main} we prove the following: 
\begin{quotation}
\emph{Let $\Sg$ be a complete orientable surface of class $C^2$ immersed in $\e$ with empty singular set and constant mean curvature $H$. If $\Sg$ is stable under a volume constraint then $H^2+\kappa\leq 0$. Moreover, if equality holds then $\Sg$ is a vertical surface.}
\end{quotation}
This theorem, together with the description of complete CMC vertical surfaces in $\e$ given in \cite{rr2} and Corollary~\ref{cor:vertmodel}, allows us to prove in Corollary~\ref{cor:main} that:
\vspace{-0,5cm}
\begin{quotation}
\emph{
\begin{itemize}
\item[i)] The vertical planes are the unique complete surfaces with empty singular set in the first Heisenberg group $\mathbb{M}(0)$ which are stable under a volume constraint.
\item[ii)] Any complete surface in the sub-Riemannian $3$-sphere $\mathbb{M}(1)$ which is stable under a volume constraint must have non-empty singular set.
\item[iii)] In the sub-Riemannian hyperbolic $3$-space $\mathbb{M}(-1)$, the mean curvature $H$ of a complete surface with empty singular set which is stable under a volume constraint satisfies $H^2\leq 1$. If equality holds, then the surface is a horocylinder. 
\end{itemize}
}
\end{quotation}

In order to prove Theorem~\ref{th:main} we first compute the second derivative of the area for certain variations of a CMC surface inside a Sasakian sub-Riemannian $3$-manifold. Second variation formulas for the area have appeared in different contexts. In the first Heisenberg group $\mathbb{M}(0)$ such formulas have been proved for minimal intrinsic graphs \cite{bscv}, \cite{mscv}, \cite{selby}, and for any minimal surface \cite{dgn}, \cite{hrr}. In \cite{chmy} and \cite{hp2}, second variation formulas are obtained for variations of class $C^3$ of CMC surfaces in pseudo-Hermitian $3$-manifolds and in vertically rigid $3$-manifolds, respectively. In \cite{montefalcone} it is computed the second derivative of the area for $C^\infty$ variations of CMC surfaces in Carnot groups.  For all the previous formulas it is assumed that the variation is supported off of the singular set. Expressions for the second derivative of the area for suitable variations moving the singular set of a $C^2$ area-stationary surface have been recently given in $\mathbb{M}(0)$, see \cite{hrr}, and in pseudo-Hermitian $3$-manifolds \cite{galli}. In Theorem~\ref{th:2ndvar} we extend the calculus developed in \cite{hrr} for minimal surfaces in $\mathbb{M}(0)$ to prove a new second variation formula supported off of the singular set. It is worth pointing out that our generalization comes not only since we consider CMC surfaces in arbitrary Sasakian sub-Riemannian $3$-manifolds, but also from the fact that the $C^1$ variations employed to move the surface may have a non-vanishing acceleration vector field, see \eqref{eq:newvar}. The motivation for choosing these type of variations comes from Lemma~\ref{lem:vpvar}: they appear in the construction of volume-preserving variations with a prescribed velocity vector field.  Our resulting formula reads
\[
(A+2HV)''(0)=\mathcal{Q}(u,u),
\]
where $A$ and $V$ are the area and volume functionals associated to the variation, $H$ is the mean curvature of the surface, $\mathcal{Q}$ is the quadratic form defined in \eqref{eq:indexform}, and $u$ is the normal component of the velocity vector field. By analogy with the Riemannian situation studied in \cite{bdce} we refer to $\mathcal{Q}$ as the \emph{index form} associated to the surface. The expression of the index form involves the derivative of $u$ along the characteristic curves, the Webster scalar curvature $K$ of the ambient manifold, and geometric terms related to the extrinsic shape of the surface. By means of a suitable integration by parts formula we also deduce that
\[
\mathcal{Q}(u,u)=-\int_{\Sg} u\,\mathcal{L}(u)\,d\Sg,
 \]
where $\mathcal{L}$ is the hypoelliptic operator on the surface $\Sg$ defined in \eqref{eq:lu}.

As a consequence of the second variation formula we prove in Proposition~\ref{prop:stcond1} that, if $\Sg$ is stable under a volume constraint, then $\mathcal{Q}(u,u)\geq 0$ for any mean zero function with compact support on $\Sg$ which is also $C^1$ along the characteristic curves. In the case of stable area-stationary surfaces the same inequality holds without assuming the mean zero condition for the function $u$. It is interesting to observe that $\mathcal{Q}$ depends on $K$ in such a way that the integral inequality $\mathcal{Q}(u,u)\geq 0$ is more restrictive provided $K\geq 0$. This is illustrated not only in Theorem~\ref{th:main}, but also in Proposition~\ref{prop:unexistence}, where we generalize the instability of the minimal Clifford torus in the sub-Riemannian $3$-sphere \cite{chmy} by showing the non-existence of compact stable area-stationary surfaces with empty singular set in Sasakian sub-Riemannian $3$-manifolds with $K\geq 1$. 

In \cite{silveira}, Da Silveira used a suitable cut-off with mean zero of the function $u=1$ to prove that a complete, non-compact, stable CMC surface $\Sg$ in $\rr^3$ is a plane. Geometrically the function $u=1$ is associated to the variation of $\Sg$ by equidistant surfaces with respect to the Euclidean distance. In the first Heisenberg group $\mathbb{M}(0)$ the same role is played by the function $u=\mnh$, where $N$ is the Riemannian unit normal to the surface for the extended metric \eqref{eq:Rm}, and $N_h$ denotes its projection onto the contact plane. In fact, in \cite{hrr} and \cite{dgnp-stable} it was possible to show, by means of a suitable cut-off of $u=\mnh$, that the vertical Euclidean planes are the unique, complete, stable area-stationary surfaces with empty singular set in $\mathbb{M}(0)$. To prove Theorem~\ref{th:main} we follow a similar argument. More precisely, we see in Proposition~\ref{prop:lnh>0} that, if the inequality $H^2+\kappa\geq 0$ holds, then the stability operator $\mathcal{L}$ in \eqref{eq:lu} satisfies $\mathcal{L}(\mnh)\geq 0$, and the inequality is strict provided $H^2+\kappa>0$, or $\Sg$ does not coincide with a vertical surface of $\e$. In both cases we are able to produce a compactly supported mean zero function $v=f\mnh$ in $\Sg$ so that $\mathcal{Q}(v,v)<0$. By Lemma~\ref{lem:vpvar} we can find a variation preserving the volume of $\Sg$ for which the area decreases. This fact, together with the stability inequality in Proposition~\ref{prop:stcond1}, completes the proof of Theorem~\ref{th:main}. It is interesting to observe that the construction of $v$ relies on the variational vector field associated to the characteristic curves of $\Sg$. Such a vector field is an example of the so-called \emph{CC-Jacobi fields} that we study in detail in Section~\ref{sec:ccgeo}.

We should stress that Corollary~\ref{cor:main} involves stable surfaces \emph{under a volume constraint}. In the particular case of minimal surfaces in the first Heisenberg group $\mathbb{M}(0)$ our characterization of the vertical planes is therefore stronger than the one given in \cite{hrr} and \cite{dgnp-stable}. As another application of Corollary~\ref{cor:main}, we deduce that the $C^2$ solutions to the isoperimetric problem in the sub-Riemannian $3$-sphere $\mathbb{M}(1)$ must have non-empty singular set. In particular, such solutions belong to the family of volume-preserving area-stationary surfaces with non-empty singular set described in \cite[Sect.~5]{hr1}. The study of the stability condition for these surfaces is treated in \cite{hr2}. On the other hand, we see that Corollary~\ref{cor:main} contains, in the case of the sub-Riemannian hyperbolic $3$-space $\mathbb{M}(-1)$, a characterization of the horocylinders as the unique complete stable surfaces under a volume constraint with empty singular set and mean curvature $H=1$. However, the classification in $\mathbb{M}(-1)$ of stable surfaces with empty singular set and $0\leq H^2<1$ remains open. It seems natural to expect, by analogy with the situation in the Riemannian hyperbolic model, that it is possible to construct examples of such surfaces different from vertical ones. 

Finally, we would like to mention that examples of area-minimizing surfaces with regularity less than $C^2$ in the first Heisenberg group $\mathbb{M}(0)$ have been found in \cite{pauls-regularity}, \cite{chy}, \cite{r2} and \cite{mscv}.  Hence our results in $\mathbb{M}(0)$ are optimal for surfaces of class $C^2$.

The paper is organized into six sections. The second one contains
some background material.  In the third one we study CC-geodesics and CC-Jacobi fields in Sasakian sub-Riemannian $3$-manifolds. In Section~\ref{sec:1ndvar} we gather some facts about area-stationary surfaces with or without a volume constraint. The fifth section is devoted to the second variation formula and the obtention of the integral inequality for stable surfaces. We conclude by proving in Section~\ref{sec:main} our existence and classification results for stable surfaces with empty singular set.

\section{Preliminaries}
\label{sec:preliminaries}
\setcounter{equation}{0}

In this section we introduce the geometric setup, and collect some results that will be needed in the paper.  We have organized it in several subsections.

\subsection{Sasakian sub-Riemannian $3$-manifolds}
\label{subsec:ssR3m}
A \emph{contact sub-Riemannian manifold} is a connected manifold $M$ together with a Riemannian metric $g_h$ defined on an oriented 
contact distribution $\h$, which is usually called \emph{horizontal distribution}. An \emph{isometry} between contact sub-Riemannian manifolds is a diffeomorphism $\phi:M\to M'$ whose differential at any $p\in M$ provides a linear isometry between $\h_p$ and $\h'_{\phi(p)}$. A contact sub-Riemannian manifold $M$ is \emph{homogeneous} if the group $\text{Iso}(M)$ of isometries of $M$ acts transitively on $M$.

Let $M$ be a contact sub-Riemannian $3$-manifold. We denote by $\eta$ the contact $1$-form on $M$ such that $\text{Ker}(\eta)=\h$ and the restriction of $d\eta$ to $\h$ coincides with the area form on $\h$. Here $d\eta$ is the $2$-form given by
\[
d\eta(U,V)=\frac{1}{2}\,\big(U(\eta(V))-V(\eta(U))-\eta([U,V])\big),
\]
where $[U,V]$ is the Lie bracket of two vector fields on $M$. We shall always choose the orientation of $M$ induced by the $3$-form $\eta\wedge d\eta$. The \emph{Reeb vector field} associated to $\eta$ is the smooth vector field $T$ transversal to $\h$ defined by $\eta(T)=1$ and $d\eta(T,U)=0$, for any vector field $U$. A vector field $U$ is \emph{horizontal} if it coincides with its horizontal projection $U_h$. A vector field is \emph{vertical} if it is proportional to $T$.

We denote by $J$ the orientation-preserving $90$ degree rotation on the contact plane $\h$. This defines a complex structure on $\h$ which is compatible with $\eta$, i.e., $J^2=-\text{Id}$, and equality
\[
g_h(U,V)=d\eta(U,J(V)) 
\]
holds for any pair of horizontal vector fields on $M$. From this we deduce that $g_h$ is a Hermitian metric, i.e., $g_h(J(U),J(V))=g_h(U,V)$. In particular, the triple $(M,\eta,-J)$ is a \emph{pseudo-Hermitian strictly pseudo-convex CR-manifold} in the sense of \cite[p.~76]{blair}. 

The \emph{canonical extension} of $g_h$ is the Riemannian metric $g=\escpr{\cdot\,,\cdot}$ on $M$ for which $T$ is a unit vector field orthogonal to $\h$. Thus $\phi\in\text{Iso}(M)$ if and only if $\phi$ is an isometry of $(M,g)$ preserving the horizontal distribution. We say that $M$ is \emph{complete} if the Riemannian manifold $(M,g)$ is complete.  

We extend $J$ to the whole tangent space to $M$ by $J(T):=0$. It is clear that
\begin{equation}
\label{eq:Rm}
\escpr{U,V}=d\eta(U,J(V))+\eta(U)\,\eta(V).
\end{equation} 
For any pair of vector fields $U$ and $V$ on $M$ it is easy to check that
\begin{align}
\label{eq:jj}
J^2(U)&=-U+\eta(U)\,T,
\\
\nonumber
\escpr{U,T}&=\eta(U),
\\
\nonumber
\escpr{J(U),J(V)}&=\escpr{U,V}-\eta(U)\,\eta(V).
\end{align}
It follows that $(M,\eta,-J,g)$ is a \emph{contact metric structure} \cite[p.~36]{blair}. Moreover, we obtain
\begin{align}
\label{eq:conmute}
\escpr{J(U),V}+\escpr{U,J(V)}&=0,
\\
\label{eq:modj}
|J(U)|^2=|U|^2-\escpr{U,T}^2&=|U_h|^2.
\end{align}

By a \emph{Sasakian sub-Riemannian $3$-manifold} we mean a contact sub-Riemannian $3$-manifold such that any diffeomorphism of the uniparametric group associated to $T$ belongs to $\text{Iso}(M)$. This is equivalent to that $T$ is a Killing field in $(M,g)$, and so the contact metric structure $(M,\eta,-J,g)$ is  \emph{K-contact}. As a consequence, the torsion introduced in \cite{chern} vanishes identically, and the Levi-Civit\`a connection $D$ associated to $g$ satisfies the equality
\begin{equation}
\label{eq:dut}
D_UT=J(U),
\end{equation}
for any vector field $U$ on $M$, see  \cite[p.~67]{blair}. In particular, any integral curve of $T$ is a geodesic in $(M,g)$ parameterized by arc-length. Furthermore, by \cite[Cor.~6.5 and Thm.~6.3]{blair}, and the second equality in \eqref{eq:jj}, we get 
\begin{equation}
\label{eq:dujv}
D_U\left(J(V)\right)=J(D_UV)+\escpr{V,T}\,U-\escpr{U,V}\,T,
\end{equation} 
for vector fields $U$ and $V$. This implies that the $CR$-structure induced by $J$ is \emph{integrable} \cite[p.~76]{blair}.

Next we gather some properties of the curvature for a Sasakian manifold. The curvature tensor in $(M,g)$ is defined by
\[
R(U,V)W=D_{V}D_{U}W-D_{U}D_{V}W+D_{[U,V]}W,
\]
for tangent vectors $U,V,W$. From \cite[Prop.~7.3]{blair} and the second equality in \eqref{eq:jj}, we have
\begin{equation}
\label{eq:ruvt}
R(U,V)T=\escpr{U,T}\,V-\escpr{V,T}\,U.
\end{equation}
The Riemannian sectional curvature of the plane generated by an orthonormal basis $\{x,y\}$ is $\escpr{R(x,y)x,y}$. By \eqref{eq:ruvt} all the planes containing $T$ have sectional curvature equal to $1$. 

We denote by $\text{Ric}$ the Ricci curvature in $(M,g)$ defined,
for tangent vectors $U$ and $V$, as the trace of the map
$W\mapsto R(U,W)V$.  By taking into account \eqref{eq:ruvt} we obtain the following equality
\begin{equation}
\label{eq:ricut}
\text{Ric}(U,T)=2\escpr{U,T}.
\end{equation}

Let $K_h(p)$ be the Riemannian sectional curvature of the contact plane $\h_p$. The \emph{Webster scalar curvature} $K$ of any contact metric structure can be computed by  $K=(1/8)(2K_h+\text{Ric}(T,T)+4)$, see \cite[Sect.~10.4]{blair}. For Sasakian $3$-manifolds we deduce from \eqref{eq:ricut} that
\begin{equation}
\label{eq:webster}
K=\frac{1}{4}\,\big(K_h+3\big).
\end{equation}
Moreover, the curvature of $(M,g)$ is completely determined by the Webster scalar curvature \cite[Thm.~7.13]{blair}. This is illustrated in the following result.

\begin{lemma}
\label{lem:ruvu}
Let $M$ be a Sasakian sub-Riemannian $3$-manifold with Webster scalar curvature $K$. For any horizontal vector $U$, and any tangent vector $V$, we have
\begin{align*}
R(U,V)U&=(4K-3)\,\escpr{V,J(U)}\,J(U)+|U|^2\,\escpr{V,T}\,T.
\\
\emph{Ric}(V,V)&=(4K-2)\,|V_h|^2+2\escpr{V,T}^2.
\end{align*}
\end{lemma}

\begin{proof}
If $U=0$ we are done. Suppose $|U|=1$. We take the orthonormal basis $\{U,W,T\}$, where $W=J(U)$. By elementary properties of the curvature tensor we obtain
\begin{align*}
R(U,V)U&=\escpr{R(U,V)U,W}\,W+\escpr{R(U,V)U,T}\,T
\\
&=\escpr{V,W}\,\escpr{R(U,W)U,W}\,W+\escpr{V,T}\,\escpr{R(U,T)U,W}\,W-\escpr{R(U,V)T,U}\,T
\\
&=K_h\,\escpr{V,W}\,W-\escpr{V,T}\,\escpr{R(U,W)T,U}\,W-\escpr{R(U,V)T,U}\,T
\\
&=(4K-3)\,\escpr{V,W}\,W+\escpr{V,T}\,T,
\end{align*}
where in the last equality we have used \eqref{eq:webster} and \eqref{eq:ruvt}.

On the other hand, if $V=V_h+\escpr{V,T}T$ then \eqref{eq:ricut} implies that
\[
\text{Ric}(V,V)=\text{Ric}(V_h,V_h)+2\escpr{V,T}^2.
\]
If $V_h=0$ we are done. Suppose $|V_h|=1$ and take the orthonormal basis $\{V_h,W,T\}$, where $W=J(V_h)$. We get
\[
\text{Ric}(V_h,V_h)=\escpr{R(V_h,W)V_h,W}+\escpr{R(V_h,T)V_h,T}=K_h+1=4K-2,
\]
where we have taken into account \eqref{eq:webster} and that the sectional curvatures of the planes containing $T$ are all equal to $1$. 
\end{proof}

\subsection{Sub-Riemannian $3$-space forms}
\label{subsec:sf}
Here we introduce the model spaces in the sub-Riemannian geometry of contact $3$-manifolds. We choose the point of view in \cite[Sect.~7.4]{blair}.

For $\kappa=-1,0,1$, we denote by $\m$ the complete, simply connected, Riemannian surface of constant sectional curvature $4\kappa$ described as follows. For $\kappa=1$ we consider the unit sphere $\mathbb{S}^2$ with its standard Riemannian metric scaled by the factor $1/4$. For $\kappa=-1,0$ we take the planar disk of radius $1/|k|$ centered at the origin and endowed with the Riemannian metric $\rho^2\,(dx^2+dy^2)$, where $\rho(x,y):=(1+\kappa(x^2+y^2))^{-1}$.  

For $\kappa=-1,0$ we denote $\e=\m\times\rr$. Let $(x,y,t)$ be Euclidean coordinates in $\rr^3$. We consider in $\e$ the contact distribution $\h=\text{Ker}(\eta)$, where $\eta:=\rho(x\,dy-y\,dx)+dt$ and $\rho(x,y,t):=\rho(x,y)$. A basis $\{X,Y,T\}$ of vector fields on $\e$ such that $X$, $Y$ are horizontal and $T$ is the Reeb vector field associated to $\eta$ is given by
\begin{align*}
X:&=\frac{1}{\rho}\left(\cos(2\kappa t)\,\frac{\ptl}{\ptl x}-\sin(2\kappa t)\,\frac{\ptl}{\ptl y}\right)
+\left(y\,\cos(2\kappa t)+x\,\sin(2\kappa t)\right)\frac{\ptl}{\ptl t},
\\
Y:&=\frac{1}{\rho}\left(\sin(2\kappa t)\,\frac{\ptl}{\ptl x}+\cos(2\kappa t)\,\frac{\ptl}{\ptl y}\right)
+\left(y\,\sin(2\kappa t)-x\,\cos(2\kappa t)\right)\frac{\ptl}{\ptl t},
\\
T:&=\frac{\ptl}{\ptl t}.
\end{align*}
On the other hand, we denote $\mathbb{M}(1)=\mathbb{S}^3$ with the planar distribution $\h=\text{Ker}(\eta)$, where $\eta:=x_1\,dy_1-y_1\,dx_1+x_2\,dy_2-y_2\,dx_1$. Here $(x_1,y_1,x_2,y_2)$ are  the Euclidean coordinates in $\rr^4$.  A basis $\{X,Y,T\}$ in the same conditions as above is given by
\begin{align*}
X:&=-x_{2}\,\frac{\ptl}{\ptl x_{1}}+y_{2}\,\frac{\ptl}{\ptl
y_{1}}+x_{1}\,\frac{\ptl}{\ptl x_{2}}-y_{1}\,\frac{\ptl}{\ptl y_{2}},
\\
Y:&=-y_{2}\,\frac{\ptl}{\ptl x_{1}}-x_{2}\,\frac{\ptl}{\ptl
y_{1}}+y_{1}\,\frac{\ptl}{\ptl x_{2}}+x_{1}\,\frac{\ptl}{\ptl y_{2}},
\\
T:&=-y_{1}\,\frac{\ptl}{\ptl x_{1}}+x_{1}\,\frac{\ptl}{\ptl y_{1}}
-y_{2}\,\frac{\ptl}{\ptl x_{2}}+x_{2}\,\frac{\ptl}{\ptl y_{2}}.
\end{align*}
In this situation we have the equalities
\begin{equation}
\label{eq:lb}
[X,Y]=-2T,\quad [X,T]=(2\kappa) Y,\quad [Y,T]=-(2\kappa) X,\quad\kappa=-1,0,1.
\end{equation}

We introduce in $\e$ the sub-Riemannian metric $g_h$ for which $\{X,Y\}$ is a positively oriented orthonormal basis at each point. The associated complex structure $J$ satisfies $J(X)=Y$ and $J(Y)=-X$. The metric $g_h$ is of \emph{bundle type} in the sense of \cite[p.~18]{montgomery}. In fact, there is a Riemannian submersion $\pi:\e\to\m$ for which $\h=(\text{Ker}(d\pi))^\bot$. Here the orthogonal complement is taken with respect to the canonical extension $g_\kappa$ of $g_h$ defined in \eqref{eq:Rm}. The map $\pi$ is  the Euclidean projection $\pi(x,y,t)=(x,y)$ if $\kappa=-1,0$, and the Hopf fibration $\pi(x_1,y_1,x_2,y_2)=(x^2_{1}+y^2_{1}-x_{2}^2-y_{2}^2,2\,(x_{2}y_{1}-x_{1}y_{2}),2\,(x_{1}x_{2}+y_{1}y_{2}))$ if $\kappa=1$. The fibers of $\pi$ are vertical geodesics in $(\e,g)$.  The diffeomorphisms associated to $T$ are translations along the fibers. As they are isometries of $g_h$ it follows that $\e$ is a Sasakian sub-Riemannian $3$-manifold. By using \eqref{eq:lb} together with the formula $K_h=4\kappa-(3/4)\escpr{[X,Y],T}^2$, see \cite[p.~187]{dcriem}, for the sectional curvature $K_h$ of the contact plane in $\e$, we deduce that $K_h=4\kappa-3$. Thus $\e$ has constant Webster scalar curvature $\kappa$ by \eqref{eq:webster}. 

\begin{remark}
In $\e$ there is a binary operation $*$ such that $(\e,*)$ is a Lie group and $\{X,Y,T\}$ is a basis of left invariant vector fields. Moreover, the canonical extension $g_\kappa$ is a left invariant metric on $\e$. The Riemannian spaces $(\e,g_\kappa)$ are important since they are examples of Thurston geometries \cite{scott} (note that, for $\kappa=1$, the metric $g_\kappa$ coincides with the standard Riemannian metric on $\mathbb{S}^3$ of constant curvature $1$). The (sub-Riemannian) isometry group $\text{Iso}(\e)$ consists of those isometries of $(\e,g_\kappa)$ that preserve the horizontal distribution. This group has dimension 4 and contains left translations and vertical Euclidean rotations.
\end{remark}

\subsection{Horizontal geometry of surfaces}
\label{subsec:surfaces}
Let $M$ be a Sasakian sub-Riemannian $3$-manifold and $\Sg$ a $C^1$ surface immersed in $M$. Since the contact plane $\h$ is a completely nonintegrable distribution it follows by Frobenius theorem that the \emph{singular set} $\Sg_0$, which consists of those points $p\in\Sg$ for which the tangent plane $T_p\Sg$ coincides with $\h_{p}$, is closed and has empty interior in $\Sg$. Hence the \emph{regular set} $\Sg-\Sg_0$ of $\Sg$ is open and dense in $\Sg$.  By using the arguments in \cite[Lem.~1]{d2}, see
also \cite[Thm.~1.2]{balogh} and \cite[Lem.~A.6]{hp2}, we can deduce that, for a $C^2$ surface $\Sg$, the Hausdorff dimension of $\Sg_{0}$ with respect to the Riemannian distance of $M$ is less than or equal to $1$.  In particular, the Riemannian area of $\Sg_{0}$ vanishes.

For a $C^1$ orientable surface $\Sg$ immersed in $M$ we denote by $N$ the unit normal vector to $\Sg$ in $(M,g)$ such that the induced orientation in $\Sg$ is compatible with the orientation of $M$ given by the $3$-form $\eta\wedge d\eta$. It is clear that
$\Sg_{0}=\{p\in\Sg;N_h(p)=0\}$, where $N_{h}=N-\escpr{N,T}T$.  In the
regular part $\Sg-\Sg_0$, we can define the \emph{horizontal Gauss
map} $\nu_h$ and the \emph{characteristic vector field} $Z$, by
\begin{equation}
\label{eq:nuh}
\nu_h:=\frac{N_h}{|N_h|}, \qquad Z=J(\nuh).
\end{equation}
As $Z$ is horizontal and orthogonal to $\nu_h$, we conclude that $Z$
is tangent to $\Sg$.  Hence $Z_{p}$ generates $T_{p}\Sg\cap\h_{p}$.
The integral curves of $Z$ in $\Sg-\Sg_0$ will be called
$(\emph{oriented}\,)$ \emph{characteristic curves} of $\Sg$.  They are
both tangent to $\Sg$ and to $\h$.  If we define
\begin{equation}
\label{eq:ese}
S:=\escpr{N,T}\,\nu_h-|N_h|\,T,
\end{equation}
then $\{Z_{p},S_{p}\}$ is an orthonormal basis of $T_p\Sg$ whenever
$p\in\Sg-\Sg_0$.  Moreover, for any $p\in\Sg-\Sg_{0}$ we have the
orthonormal basis of $T_{p}M$ given by $\{Z_{p},(\nuh)_{p},T_{p}
\}$.  From here we deduce the following identities on $\Sg-\Sg_{0}$
\begin{equation}
\label{eq:relations}
|N_{h}|^2+\escpr{N,T}^2=1, \qquad (\nu_{h})^\top=\escpr{N,T}\,S, 
\qquad T^\top=-|N_{h}|\,S,
\end{equation}
where $U^\top$ stands for the projection of a vector field $U$ onto
the tangent plane to $\Sg$.

Given a $C^1$ orientable surface $\Sg$ immersed in $M$, we
define the \emph{area} of $\Sg$ by
\begin{equation}
\label{eq:area}
A(\Sg):=\int_{\Sg}|N_{h}|\,d\Sg,
\end{equation}
where $d\Sg$ is the Riemannian area element on $\Sg$.  If $\Sg$ is a
$C^2$ surface bounding a set $\Om\subset M$ then, as a consequence of the Riemannian divergence theorem, we deduce that $A(\Sg)$ coincides with the \emph{sub-Riemannian perimeter \'a la De Giorgi} of $\Om$ defined following \cite{cdg1} by
\begin{equation}
\label{eq:per}
P(\Om):=\sup\left\{\int_\Om\divv U\,dM;\,|U|\leq 1\right\}.
\end{equation}
Here the supremum is taken over $C^1$ horizontal vector fields on $M$. In the definition above $dM$ and $\divv$ are the Riemannian volume and divergence, respectively.  

Finally, for a $C^2$ orientable surface $\Sg$, we denote by $B$ the Riemannian shape operator of $\Sg$.  It is defined, for any vector $U$ tangent to $\Sg$, by $B(U)=-D_{U}N$.  The Riemannian mean curvature of $\Sg$ is
$-2H_{R}=\divv_{\Sg}N$, where $\divv_{\Sg}$ denotes the Riemannian
divergence relative to $\Sg$.

\section{Carnot-Carath\'eodory geodesics and Jacobi fields}
\label{sec:ccgeo}

Let $M$ be a contact sub-Riemannian $3$-manifold. The planar distribution $\h$ is a bracket-generating distribution of constant step $2$ since it coincides with the kernel of a contact $1$-form. Let $\ga:I\to\stres$ be a piecewise $C^1$ curve defined on a compact interval $I\sub\rr$.  The \emph{length} of $\ga$ is the Riemannian
length $L(\ga):=\int_{I}|\dot{\ga}(\eps)|\,d\eps$, where $\dot{\ga}$ is the tangent vector of $\ga$. A \emph{horizontal curve} in $M$ is a curve $\ga$ such that $\dot{\ga}$ lies in $\h$ wherever it exists. For two given points in $M$ we can find by Chow's connectivity theorem \cite[Sect.~1.2.B] {gromov-cc} a $C^\infty$ horizontal curve joining these points.  The \emph{Carnot-Carath\'eodory distance} $d_{cc}$ between two points in $M$ is defined as the infimum of the lengths of all $C^\infty$ horizontal curves
joining the given points.  The topology associated to $d_{cc}$ is the original topology of $M$, see \cite[Cor.~2.6]{andre}. The Hausdorff dimension of $(M,d_{cc})$ is $4$ by Mitchell's measure theorem \cite[Thm.~2.17]{montgomery}.  In the metric space $(\stres,d_{cc})$ there is a natural extension for continuous curves of the notion of length, see \cite[p.~19]{andre}. A continuous curve $\ga$ joining $p,q\in M$
is \emph{length-minimizing} if $L(\ga)=d_{cc}(p,q)$.  When the
metric space $(\stres,d_{cc})$ is complete we can apply the Hopf-Rinow
theorem in sub-Riemannian geometry \cite[Thm.~2.7]{andre}, \cite[Thm.~1.19]{montgomery} to ensure the existence of absolutely continuous horizontal length-minimizing curves joining two given points.
Moreover,  any of these minimizers, suitably parameterized, is a $C^\infty$ curve satisfying the normal geodesic equations \cite[Cor.~6.2]{strichartz}, \cite[Sect.~5.6]{montgomery}. In particular, a length-minimizing curve can be parameterized by a $C^2$ critical point of length under variations by horizontal curves with fixed endpoints. We refer to these curves as \emph{Carnot-Carath\'eodory geodesics}. In this section we gather some facts about geodesics that will be helpful in the remainder of the paper. 

The geodesic equation is usually obtained by means of variational arguments, see for example \cite[Prop.~15]{rumin} and \cite[Sect.~3]{hughen}. Here we follow the point of view given in \cite[Sect.~3]{rr2} and \cite[Sect.~3]{hr1}. Let $M$ be a Sasakian sub-Riemannian $3$-manifold. By a \emph{CC-geodesic} of curvature $\lambda$ we mean a $C^2$ horizontal curve $\gamma$, parameterized by arc-length, and satisfying the  equation
\begin{equation}
\label{eq:geoeq}
\dot{\ga}'+2\la\,J(\dot{\ga})=0,
\end{equation}
where the prime $'$ stands for the Riemannian covariant derivative along $\ga$. In particular, CC-geodesics of vanishing curvature are horizontal Riemannian geodesics in $(M,g)$. It can be proved as in \cite[Prop.~3.1]{rr2} that a CC-geodesic is a stationary point of the length functional associated to any variation by horizontal curves with the same endpoints. Moreover, any sufficiently short arc of a CC-geodesic is length-minimizing \cite[Thm.~1.14]{montgomery}. If $p\in M$ and $v$ is a unit vector in $\h_p$ then the unique solution $\ga$ to \eqref{eq:geoeq} with $\ga(0)=p$ and $\dot{\ga}(0)=v$ is a CC-geodesic of curvature $\la$ since the functions $\escpr{\dot{\ga},T}$ and $|\dot{\ga}|^2$ are constant along $\ga$. If $M$ is complete then the CC-geodesics are defined on the whole real line by \cite[Thm.~1.2]{falbel4}.

The next lemma is a characterization of the CC-geodesics in a Sasakian $3$-manifold whose sub-Riemannian structure is of bundle type. 

\begin{lemma}
\label{lem:geofunction}
Let $\ga$ be a $C^2$ horizontal curve parameterized by
arc-length inside a Sasakian sub-Riemannian $3$-manifold $M$. Suppose that $\pi:M\to E$ is a Riemannian submersion onto a Riemannian surface $E$ such that $\h=(\emph{Ker}(d\pi))^\bot$. Then $\ga$ is a $CC$-geodesic of curvature $\la$ if and only if the curve $\alpha:=\pi\circ\ga$ has constant geodesic curvature $h=2\la$ in $E$ with respect to the unit normal $n:=-(d\pi)(J(\dot{\ga}))$.
\end{lemma}

\begin{proof}
By \eqref{eq:conmute} and \eqref{eq:modj} we have that $\{\dot{\ga},J(\dot{\ga})\}$ is an orthonormal basis of the contact plane of $M$ along $\ga$. Hence $\{\dot{\alpha},n\}$ is an orthonormal basis of the tangent plane to $E$ along $\alpha$. We denote by primes $'$ the Riemannian covariant derivatives along $\gamma$ and $\alpha$. The vector field $\dot{\ga}'$ is proportional to $J(\dot{\ga})$ since $\escpr{\dot{\ga},T}=0$ and $|\dot{\ga}|^2=1$. By \cite[p.~186]{dcriem} we deduce that $\dot{\alpha}'=(d\pi)(\dot{\ga}')$. Therefore, we get
\[
h=\escpr{\dot{\alpha}',n}=-\escpr{\dot{\ga}',J(\dot{\ga})}.
\]
The claim follows by taking into account the previous equality and equation \eqref{eq:geoeq}.
\end{proof}

By using Lemma~\ref{lem:geofunction} or \cite[Thm.~1.26]{montgomery} we can compute  the CC-geodesics inside the sub-Riemannian space forms $\e$ introduced in Section~\ref{subsec:sf}. In the first Heisenberg group $\mathbb{M}(0)$ it is well known that a complete CC-geodesic is either a horizontal straight line or a Euclidean helix of vertical axis \cite{montgomery}, \cite{survey}, \cite[Sect.~3]{rr2}. In the sub-Riemannian $3$-sphere $\mathbb{M}(1)$, depending on an arithmetic condition involving $\la$, all the complete CC-geodesics of curvature $\la$ are either closed circles of the same length, or they parameterize dense subsets of a Clifford torus \cite[Prop.~3.3]{hr1}. Explicit expressions for the CC-geodesics in the special linear group $\text{SL}(2,\rr)$, which is a non-simply connected sub-Riemannian hyperbolic model, can be found in \cite[Sect.~7]{markina} and \cite[Sect.~3.3]{boscain-rossi}. For our purposes in Section~\ref{sec:main} we shall only need the following fact about CC-geodesics in $\e$.

\begin{lemma}
\label{lem:geomodel}
In a model space $\e$ all the complete CC-geodesics of the same curvature are either injective curves defined on $\rr$, or closed circles with the same length.
\end{lemma}

\begin{proof}
By the notes previous to the lemma it suffices to see that a complete geodesic $\ga:\rr\to\mathbb{M}(-1)$ is an injective curve. Let $\pi:\mathbb{M}(-1)\to\mathbb{N}(-1)$ be the Riemannian submersion given by $\pi(x,y,t)=(x,y)$. We know by
Lemma~\ref{lem:geofunction} that $\alpha:=\pi\circ\ga$ is a complete curve with constant geodesic curvature in $\mathbb{N}(-1)$. If $\alpha$ is either an open arc of a circle, or a line approaching $\ptl\mathbb{N}(-1)$, then the claim is clear. Suppose that $\alpha$ is a circle contained in $\mathbb{N}(-1)$. After an isometry of $\mathbb{M}(-1)$ we may assume that the circle in centered at the origin. The fact that $\ga(s)=(x(s),y(s),t(s))$ is horizontal is equivalent to the equation $\dot{t}=\rho\,(\dot{x}y-x\dot{y})$ along $\ga$, where $\rho(x,y)=(1-x^2-y^2)^{-1}$. An easy computation shows that $t(s)$ is a strictly monotonic function, and that $\ga$ is a Euclidean helix about the $t$-axis.    
\end{proof} 

It was proved in \cite[Lem.~3.5]{rr2} that the infinitesimal vector field
associated to a $C^2$ variation by CC-geodesics of the same curvature in the first Heisenberg group satisfies a second order differential equation similar to the classical Jacobi equation in Riemannian geometry.  Next we obtain the same result in a more general context, and under weaker regularity conditions.

\begin{lemma}
\label{lem:ccjacobi}
Let $M$ be a Sasakian sub-Riemannian $3$-manifold. Consider a $C^1$ curve $\alpha:I\to M$ defined on some open interval $I\subseteq\rr$, and a $C^1$ unit horizontal vector field $U$ along $\alpha$. For fixed $\lambda\in\rr$, suppose that we have a well-defined map $F:I\times I'\to M$ given by $F(\eps,s):=\ga_{\eps}(s)$, where $I'$ is an open interval containing the origin, and $\ga_{\eps}(s)$ is the CC-geodesic of curvature $\la$ with initial conditions $\ga_{\eps}(0)=\alpha(\eps)$ and $\dot{\ga}_{\eps}(0)=U(\eps)$. Then, the vector field $V_{\eps}(s):=(\ptl F/\ptl\eps)(\eps,s)$ satisfies the following properties:
\begin{itemize}
\item[(i)] $V_\eps$ is $C^\infty$ along $\ga_\eps$ with $[\dot{\ga}_\eps,V_\eps]=0$,
\item[(ii)] the function $\la\,\escpr{V_\eps,T}+\escpr{V_\eps,\dot{\ga}_\eps}$
is constant along $\ga_\eps$, 
\item[(iii)] $V_\eps$ satisfies the $CC$-Jacobi equation
\[
V_{\eps}''+R(\dot{\ga}_{\eps},V_{\eps})\dot{\ga}_{\eps}
+2\la\,\big(J(V'_{\eps})-\escpr{V_{\eps},\dot{\ga}_{\eps}}\,T\big)=0, 
\]
where $R$ is the Riemannian curvature tensor in $(M,g)$ and the prime $'$ denotes the covariant derivative along the geodesic $\ga_{\eps}$,  
\item[(iv)] the vertical component of $V_\eps$ satisfies the differential equation
\[
\escpr{V_\eps,T}'''+4(\lambda^2+K)\,\escpr{V_\eps,T}'=0,
\]
where $K$ is the Webster scalar curvature of $M$ and $'$ is the derivative with respect to $s$.  
\end{itemize}
\end{lemma}

\begin{proof}
For fixed $\lambda\in\rr$  the unique solution of \eqref{eq:geoeq} with initial conditions $p\in M$ and $v\in\h_p$ depends differentiably ($C^\infty$) on the initial data. Hence the geodesic flow $\mathcal{F}$ associated to CC-geodesics of curvature $\la$ is a $C^\infty$ map on the horizontal bundle of $M$.  It follows that $F$ is $C^1$ on $I\times I'$ and $C^\infty$ with respect to $s$. That $(\ptl F/\ptl \eps)(\eps,s)$ is $C^\infty$ with respect to $s$ comes from the differentiability of $\mathcal{F}$. Moreover, for any $k\in\mathbb{N}$ and any component $\phi$ of $F$ in local coordinates, the derivatives
$\ptl^{k+1}\phi/\ptl s^k\,\ptl\eps$ and $\ptl^k\phi/\ptl s^k$ exist
and are continuous functions on $I\times I'$. In particular, an application of the Schwarz's theorem implies that $\ptl^2\phi/\ptl\eps\ptl
s=\ptl^2\phi/\ptl s\ptl\eps$ and $\ptl^3\phi/\ptl
s^2\ptl\eps=\ptl^3\phi/\ptl s\ptl\eps\ptl s=\ptl^3\phi/\ptl\eps\ptl
s^2$. Having this in mind the proofs in \cite[p.~68 and p.~98]{dcriem} can be
traced to get $[\dot{\ga}_{\eps},V_{\eps}]=0$ and
$R(\dot{\ga}_{\eps},V_{\eps})\dot{\ga}_{\eps}=D_{V_{\eps}}\dot{\ga}'_{\eps}
-V_\eps''$. Now we can prove statements (ii) and (iii) following the arguments in \cite[Lem. ~3.4 and Lem.~3.5]{rr2} together with equalities \eqref{eq:dut}, \eqref{eq:conmute}, \eqref{eq:geoeq} and \eqref{eq:dujv}. 

Finally we prove (iv). For simplicity we avoid the subscript $\eps$ in the computations below.  Note that
\[
\escpr{V',T}=\escpr{D_{\dot{\ga}}V,T}=\escpr{D_{V}\dot{\ga},T}=
-\escpr{\dot{\ga},J(V)}=\escpr{V,J(\dot{\ga})},
\]
where we have used $[\dot{\ga},V]=0$, \eqref{eq:dut} and  \eqref{eq:conmute}. As $T'=J(\dot{\ga})$, we deduce
\begin{equation}
\label{eq:prima1}
\escpr{V,T}'=\escpr{V',T}+\escpr{V,T'}=2\escpr{V,J(\dot{\ga})}.
\end{equation}
On the other hand, an application of \eqref{eq:dujv}, \eqref{eq:geoeq} and 
\eqref{eq:jj} yields
\[
J(\dot{\ga})'=D_{\dot{\ga}}(J(\dot{\ga}))=J(\dot{\ga}')-T=2\la\dot{\ga}-T,
\]
and so, if we differentiate in \eqref{eq:prima1}, then we obtain
\begin{equation}
\label{eq:prima21}
\escpr{V,T}''=2\left(\escpr{V',J(\dot{\ga})}+2\la\,\escpr{V,\dot{\ga}}-\escpr{V,T}\right).
\end{equation}
Now we differentiate in \eqref{eq:prima21} to get
\[
(1/2)\escpr{V,T}'''=\escpr{V'',J(\dot{\ga})}+\escpr{V',J(\dot{\ga})'}+2\la\,\escpr{V,\dot{\ga}}'-\escpr{V,T}'.
\]
By taking into account \eqref{eq:prima1} together with statements (iii) and (ii), we deduce
\begin{align*}
(1/2)\escpr{V,T}'''&=-\escpr{R(\dot{\ga},V)\dot{\ga},J(\dot{\ga})}-2\la\escpr{J(V'),J(\dot{\ga})}
+2\la\,\escpr{V',\dot{\ga}}-\escpr{V',T}
\\
&\quad\,-4\la^2\,\escpr{V,J(\dot{\ga})}-2\escpr{V,J(\dot{\ga})}
\\
&=(3-4K)\,\escpr{V,J(\dot{\ga})}-4\la^2\,\escpr{V,J(\dot{\ga})}-3\escpr{V,J(\dot{\ga})}
\\
&=-4(\la^2+K)\,\escpr{V,J(\dot{\ga})}=-2(\la^2+K)\,\escpr{V,T}'.
\end{align*}
To obtain the second equality we have used Lemma~\ref{lem:ruvu} and formula \eqref{eq:conmute}.
\end{proof}

If $\ga$ is a CC-geodesic in $M$ then we define a \emph{CC-Jacobi field} along $\ga$ as a solution to the CC-Jacobi equation in Lemma~\ref{lem:ccjacobi} (iii). If the Webster scalar curvature of $M$ is constant along $\ga$ then, an easy integration from Lemma~\ref{lem:ccjacobi} (iv), gives us the following result.

\begin{lemma}
\label{lem:jacobicon}
Let $\ga(s)$ be a $CC$-geodesic of curvature $\la$ in a Sasakian sub-Riemannian $3$-manifold $M$. Let $V$ be the $CC$-Jacobi field associated to a variation of $\ga$ by $CC$-geodesics of curvature $\la$ as in Lemma~\ref{lem:ccjacobi}. Suppose that the Webster scalar curvature $K$ of $M$ is constant along $\ga$. If we denote $\mu:=4(\la^2+K)$ and $f(s):=\escpr{V,T}(s)$, then we have
\begin{itemize}
\item[(i)] For $\mu<0$, 
$$f(s)=\frac{1}{\sqrt{-\mu}}\left(a\sinh(\sqrt{-\mu}\,s)+b\cosh(\sqrt{-\mu}\,s)\right)+c,$$
 where $a=f'(0)$, $b=(1/\sqrt{-\mu})f''(0)$ and $c=f(0)+(1/\mu)f''(0)$.
\item[(ii)] For $\mu=0$, 
$$f(s)=as^2+bs+c,$$
where $a=(1/2)f''(0)$, $b=f'(0)$ and $c=f(0)$.
\item[(iii)] For $\mu>0$,   
$$f(s)=\frac{1}{\sqrt{\mu}}\big(a\sin(\sqrt{\mu}\,s)-b\cos(\sqrt{\mu}\,s)\big)+c,$$
where $a=f'(0)$, $b=(1/\sqrt{\mu})f''(0)$ and $c=f(0)+(1/\mu)f''(0)$.
\end{itemize}
\end{lemma}

\begin{remark}
\label{re:mm}
Previous work on the second variation of length and CC-Jacobi fields in contact sub-Riemannian $3$-manifolds is found in \cite{hughen} and \cite{rumin}. The map $F$ in Lemma~\ref{lem:ccjacobi} provides a family of CC-geodesics of the same curvature. The CC-Jacobi equation for arbitrary $C^2$ variations by horizontal curves of a given CC-geodesic is obtained in \cite[Sect.~3.2]{hughen} and \cite[Lem.~1.4]{chanillo-yang}. The differential equation in Lemma~\ref{lem:ccjacobi} (iv) appears in \cite[Lem.~1.5]{chanillo-yang} when the authors study the CC-Jacobi fields associated to $C^2$ variations by CC-geodesics leaving from the same point  and such that either the curvature or the initial velocity are variable. In \cite[Lem.~2.1]{chanillo-yang} there is an explicit computation of the vertical components of such CC-Jacobi fields when $K$ is constant along the geodesic. Recently Ritor\'e  \cite{posreach} has also used CC-Jacobi fields to study curvature measures in the Heisenberg groups of higher dimension.
\end{remark}

\section{Area-stationary surfaces in Sasakian $3$-manifolds}
\label{sec:1ndvar}

In this section we introduce the variational background that we will follow in the paper. We also study critical surfaces for the area functional \eqref{eq:area} with or without a volume constraint.

Let $M$ be a Sasakian sub-Riemannian $3$-manifold and $\Sg$ a $C^2$ orientable surface immersed in $M$. Recall that $N$ denotes the unit normal vector to $\Sg$ in $(M,g)$ such that the induced orientation in $\Sg$ is compatible with the orientation of $M$ given by the $3$-form $\eta\wedge d\eta$.

By a \emph{variation} of $\Sg$ we mean a $C^1$ map
$\varphi:I\times\Sg\to M$, where $I\subeq\rr$ is an open interval
containing the origin, and $\varphi$ satisfies:
\begin{itemize}
\item[(i)] $\var(0,p)=p$ for any $p\in\Sg$,
\item[(ii)] the set $\Sg_{s}:=\{\varphi(s,p);\,p\in\Sg\}$ is a $C^1$ 
surface immersed in $M$ for any $s\in I$, 
\item[(iii)] the map $\varphi_{s}:\Sg\to\Sg_{s}$ given
by $\varphi_{s}(p):=\varphi(s,p)$ is a diffeomorphism for any $s\in I$,
\item[(iv)] $\varphi$ is compactly supported: there is a compact set $C\subseteq\Sg$ such that $\varphi_{s}(p)=p$ for any $s\in I$ and any $p\in\Sg-C$.
\end{itemize}
The area functional associated to the variation is $A(s):=A(\Sg_{s})$. If $\Sg$ is embedded and encloses a region $\Om\sub M$ then the variation produces a family $\Om_s\sub M$ such that $\Om_0=\Om$ and $\ptl\Om_s=\Sg_s$ for any $s\in I$. In this situation the volume functional is given by $V(s):=V(\Om_s)$, where $V(\cdot)$ stands for the Riemannian volume in $(M,g)$. In general, for an immersed surface $\Sg$, we define the volume functional as in \cite[Sect.~2]{bdce}, by
\[
V(s)=\int_{[0,s]\times C}\varphi^*(dM),
\]
which represents the signed volume enclosed between $\Sg$ and $\Sg_s$.
We say that the variation is \emph{volume preserving} if $V(s)$ is
constant for any $s$ small enough. We say that $\Sg$ is area-stationary if $A'(0)=0$ for any variation of $\Sg$. We say that $\Sg$ is
\emph{volume-preserving area-stationary} or \emph{area-stationary
under a volume constraint} if $A'(0)=0$ for any volume-preserving
variation of $\Sg$. 

Now we describe the main variations that will be used in the paper. Let $\om:I\times\Sg\to\rr$ be a $C^1$ function defined on the product of an open interval $I$ containing the origin and a $C^2$ surface $\Sg$. We set  $(\ptl\om/\ptl p)(s,p)=(\nabla_{\Sg}\,\om_{s})_{p}$, where $\om_{s}(p)=\om(s,p)$ and $\nabla_{\Sg}$ is the gradient relative to $\Sg$ in $(M,g)$. For fixed $p\in\Sg$, the map $s\to (\ptl\om/\ptl p)(s,p)$ is a continuous curve in the tangent plane $T_p\Sg$. If this curve is of class $C^k$ then we denote the derivative of order $k$ at $s$ by $(\ptl^{k+1}\om/\ptl s^k\ptl p)(s,p)$. We suppose that $\om$ satisfies the following properties:
\begin{align}
\label{eq:omega1}
&\om(0,p)=0, \text{ for any } p\in\Sg,
\\
\label{eq:omega2}
&\ptl^2\om/\ptl s^2 \text{ exists and it is continuous on } 
I\times\Sg,
\\
\label{eq:omega3}
&\ptl^{k+1}\om/\ptl s^k\ptl p \, \text{ exists and it is continuous 
on } I\times\Sg  \text{ for } k=1,2, 
\\
\label{eq:omega4}
&\text{there is a compact set } C\subseteq\Sg \, \text{ such that } 
\om(s,p)=0, \text{ for any } s\in I \text{ and any } p\in\Sg-C.
\end{align}
Under the conditions above we can apply Schwarz's theorem as in the proof of Lemma~\ref{lem:ccjacobi} to deduce that the functions $u(p):=(\ptl\om/\ptl s)(0,p)$ and $w(p):=(\ptl^2\om/\ptl
s^2)(0,p)$ belong to $C^1_{0}(\Sg)$.  Given a $C^1$ vector field $Q$
on $\Sg$ we define, for $s\in I$ small enough, the map
\begin{equation}
\label{eq:newvar}
\varphi_{s}(p):=\exp_{p}(\omega(s,p)\,Q_{p}),
\end{equation}
where $\exp_{p}$ is the
exponential map of $(M,g)$ at $p$.  This is a variation of $\Sg$.  In fact, if $c_{p}$ denotes the geodesic in $(M,g)$
with $c_{p}(0)=p$ and $\dot{c}_{p}(0)=Q_{p}$ then it is clear that
$\varphi_{s}(p)=c_{p}(\om(s,p))$. Hence the associated
velocity and acceleration vectors of the variation are the $C^1$ vector fields with compact support on $\Sg$ given by
\[
U_p=\frac{d}{ds}\bigg|_{s=0}\varphi_s(p)=u(p)\,Q(p),\qquad
W_p=\frac{d^2}{ds^2}\bigg|_{s=0}\varphi_s(p)=w(p)\,Q(p).
\]
If we take $\om(s,p)=s$ in \eqref{eq:newvar} then we obtain the variations used in \cite{rr2}. They satisfy $U_p=Q_p$ and $W_p=0$. In general, when we have to construct a volume-preserving variation with a prescribed velocity we may expect that the associated acceleration vector need not vanish. This is shown in the next result, which provides the main motivation to introduce the variations in \eqref{eq:newvar}, see \cite[Lem.~2.4]{bdc} and \cite[Lem.~2.2]{bdce} for a proof. 

\begin{lemma}
\label{lem:vpvar}
Let $\Sg$ be a $C^2$ orientable surface immersed in a Sasakian sub-Riemannian $3$-manifold. For any $u\in C^1_0(\Sg)$ with $\int_\Sg u\,d\Sg=0$ there is a volume-preserving variation as in \eqref{eq:newvar} with $Q_p=N_p$, support contained in 
$\emph{supp}(u)$, and velocity vector field $U=uN$. 
\end{lemma}

\subsection{Stationary surfaces}
Here we gather some basic properties of critical points of the area with or without a volume constraint that will be useful in the remainder of the paper.

Let $M$ be a Sasakian sub-Riemannian $3$-manifold and $\Sg$ an orientable $C^2$ surface immersed in $M$. The (\emph{sub-Riemannian}) \emph{mean
curvature} of $\Sg$ is the function defined in \cite{rr1} and \cite{rr2} by equality
\begin{equation}
\label{eq:mc}
-2H(p)=(\divv_{\Sg}\nu_{h})(p),\qquad p\in\Sg-\Sg_{0},
\end{equation}
where $\nuh$ is the horizontal Gauss map introduced in \eqref{eq:nuh} and
$\divv_{\Sg}U$ stands for the divergence relative to $\Sg$ in $(M,g)$ of a $C^1$
vector field $U$.  

The following equalities are obtained by reproducing the computations in \cite[Lem.~3.5 and Rem.~3.6]{hrr} with the help of \eqref{eq:dut} and \eqref{eq:conmute}.

\begin{lemma}
Let $\Sg$ be a $C^2$ orientable surface immersed in a Sasakian sub-Riemannian $3$-manifold $M$.  Consider a point $p\in\Sg-\Sg_0$, the horizontal Gauss map $\nu_h$ and the characteristic field $Z$ defined in \eqref{eq:nuh}.  For any $v\in
T_{p}M$ we have
\begin{align}
\label{eq:vmnh}
v\,(|N_h|)&=\escpr{D_{v}N, \nu_{h}}-\escpr{N,T}\,
\escpr{J(v),\nu_{h}},
\\
\label{eq:vnt}
v(\escpr{N,T})&=\escpr{D_{v}N,T}+\escpr{N,J(v)}, \\
\label{eq:dvnuh}
D_{v}\nu_h&=|N_h|^{-1}\, \big(\escpr{D_vN,Z}-\escpr{N,T}\,
\escpr{J(v),Z}\big)\,Z+\escpr{Z,v}\,T.
\end{align}
Moreover, if $H$ is the mean curvature of $\Sg$, then we get the following identities on $\Sg-\Sg_{0}$
\begin{align}
\label{eq:mc2}
2H&=|N_{h}|^{-1}\,\escpr{B(Z),Z},
\\
\label{eq:dzz}
D_{Z}Z&=(2H)\,\nuh,
\end{align}
where $B$ is the Riemannian shape operator of $\Sg$.
\end{lemma}

We say that $\Sg$ is a \emph{constant mean curvature surface} (CMC surface) if $H$ is constant on $\Sg-\Sg_{0}$. A CMC surface with $H=0$ is a \emph{minimal surface}. 
CMC surfaces satisfy nice geometric and analytical properties. They appear as critical points of the area for volume-preserving variations supported off of the singular set $\Sg_0$. They are also ruled surfaces of $M$. These facts are well known and have been studied in pseudo-Hermitian manifolds \cite{chmy}, vertically rigid manifolds \cite{hp1}, and even in general sub-Riemannian manifolds \cite{hp2}. 

\begin{proposition}
\label{prop:varprop}
Let $\Sg$ be a $C^2$ orientable surface immersed inside a Sasakian sub-Riemannian $3$-manifold $M$.  Then we have
\begin{itemize}  
\item[(i)] If $\Sg$ is area-stationary $($resp.~volume-preserving area-stationary$)$ then $\Sg$ is minimal $($resp.~$\Sg$ is a CMC surface$)$.
\item[(ii)] $\Sg$ has constant mean curvature $H$ if and only if $(A+2HV)'(0)=0$ for
any variation of $\Sg$ as in \eqref{eq:newvar} supported in $\Sg-\Sg_0$.
\item[(iii)] If $\Sg$ has constant mean curvature $H$ then any characteristic curve of $\Sg$ is an open piece of a CC-geodesic in $M$ of curvature $H$.  
\item[(iv)] If $\Sg$ is a CMC surface then, in $\Sg-\Sg_{0}$, the normal vector $N$ is $C^\infty$ in the 
direction of the characteristic field $Z$.
\end{itemize}
\end{proposition}

\begin{remark}
\label{re:divv}
A continuous vector field $V$ (resp.~a continuous 
function $f$) on $\Sg-\Sg_{0}$ is $C^\infty$ in the $Z$-direction if 
the restriction of $V$ (resp. $f$) to any characteristic curve $\ga$ 
provides a $C^\infty$ vector field (resp. $C^\infty$ function) along 
$\ga$. If $f$ is $C^1$ in the $Z$-direction then we define
$\divv_\Sg(fZ):=f\divv_\Sg Z+Z(f)$, see \cite[Sect.~2.7]{hrr} for details.
\end{remark}

\begin{proof}[Proof of Proposition~\ref{prop:varprop}]
Take a variation of $\Sg$ as in \eqref{eq:newvar} with support contained in $\Sg-\Sg_0$. Following the proof of \cite[Lem.~4.3]{rr2} we can show that the first derivative of the area functional for such a variation is given by
\begin{equation}
\label{eq:aprima}
A'(0)=-\int_\Sg 2H\,\escpr{U,N}\,d\Sg,
\end{equation}
where $U$ is the associated velocity vector field. On the other hand, it is known \cite[Lem.~2.1]{bdce} that the first derivative of the volume functional is
\begin{equation}
\label{eq:vprima}
V'(0)=\int_\Sg\escpr{U,N}\,d\Sg.
\end{equation}
For any function $u\in C^1_0(\Sg-\Sg_0)$ with $\int_\Sg u\,d\Sg=0$ we can apply Lemma~\ref{lem:vpvar} to find a volume-preserving-variation as in \eqref{eq:newvar}, supported in $\Sg-\Sg_0$, and whose velocity vector field is $uN$. If $\Sg$ is volume-preserving area-stationary then $\int_\Sg Hu\,d\Sg=0$ for any mean zero function $u$ with compact support in $\Sg-\Sg_0$. This proves (i) for volume-preserving area-stationary surfaces. That an area-stationary surface is minimal and statement (ii) follow  from \eqref{eq:aprima} and \eqref{eq:vprima} by using variations as in \eqref{eq:newvar} with $\om(s,p)=s\,u(p)$ and $Q_p=N_p$. Statement (iii) is a consequence of \eqref{eq:geoeq} and~\eqref{eq:dzz}.

To prove (iv) we proceed as in
\cite[Lem.~3.4]{hrr}.  Take $p\in\Sg-\Sg_{0}$.  We denote by $\ga$ the
characteristic curve of $\Sg$ through $p$.  Let
$\alpha:(-\eps_{0},\eps_{0})\to\Sg-\Sg_{0}$ be a $C^1$ curve
transverse to $\ga$ with $\alpha(0)=p$.  Define, for $s$ small enough, the map
$F(\eps,s):=\ga_{\eps}(s)$, where $\ga_{\eps}$ is the characteristic curve passing through $\alpha(\eps)$. We deduce from statement (iii) that any $\ga_{\eps}$ is a CC-geodesic in $M$ of curvature $H$.  By Lemma~\ref{lem:ccjacobi} (i) we get that $V(s):=(\ptl F/\ptl\eps)(0,s)$ is
a $C^\infty$ vector field along $\ga$ which is also tangent to $\Sg$.  Since both $\dot{\ga}(s)$ and $V(s)$ are $C^\infty$ and linearly independent for $s$ small enough,
the unit normal $N$ to $\Sg$ along $\ga$ is given by $N=\pm (|V|^2-\escpr{V,\dot{\ga}}^2)^{-1/2}\,\left (\escpr{V,T}\,J(\dot{\ga})-\escpr{V,J(\dot{\ga})}\,T\right)$.
\end{proof}

\subsection{Vertical surfaces}
\label{subsec:vertical}
Here we study certain CMC surfaces that will play an important role in our classification results in Section~\ref{sec:main}.

Let $\Sg$ be a $C^2$ orientable surface immersed in a Sasakian sub-Riemannian $3$-manifold $M$. We say that $\Sg$ is \emph{vertical} if the Reeb vector field $T$ is always tangent to $\Sg$. Clearly a vertical surface has empty singular set. 

Complete CMC vertical surfaces have been described in the first Heisenberg group \cite[Lem.~4.9]{gp}, \cite[Prop.~6.16]{rr2} and in the sub-Riemannian $3$-sphere \cite[Prop.~5.11]{hr1}. We generalize these results by showing that a complete CMC vertical surface inside a Sasakian $3$-manifold $M$, whose sub-Riemannian structure is of bundle type, is a cylinder over a curve of constant geodesic curvature in the base space. Recall that a surface $\Sg$ immersed in $M$ is \emph{complete} if it is complete with respect to the Riemannian metric induced by $(M,g)$.

\begin{proposition}
\label{prop:vert}
Let $M$ be a complete Sasakian sub-Riemannian $3$-manifold. Then we have
\begin{itemize}
\item[(i)] If $\ga:\rr\to M$ is a CC-geodesic of curvature $\la$, then the map $F:\rr^2\to M$ given by $F(\eps,s)=\tau_\eps(s)$, where $\tau_\eps(s)$ is the integral curve of $T$ with $\tau_\eps(0)=\ga(\eps)$, defines a complete vertical surface $\Sg_\ga$ of constant mean curvature $\la$.
\item[(ii)] Any complete, connected, CMC vertical surface $\Sg$ in $M$ is a surface $\Sg_\ga$ as obtained in \emph{(i)}.
\item[(iii)] Suppose that $\pi:M\to E$ is a Riemannian submersion onto a Riemannian surface $E$ such that $\h=(\emph{Ker}(d\pi))^\bot$. If $\Sg$ is a complete vertical surface of constant mean curvature $H$ then $\Sg=\pi^{-1}(\alpha)$, where $\alpha$ is a curve in $E$ of constant geodesic curvature $2H$ computed with respect to the unit normal $n:=(d\pi)(N)$. 
\end{itemize}
\end{proposition}

\begin{proof}
Let $\{\varphi_s\}_{s\in\rr}$ be the uniparametric group of diffeomorphisms associated to $T$. As $M$ is Sasakian then any $\varphi_s$ is an isometry of $(M,g)$. We have $(\ptl F/\ptl \eps)(\eps,s)=(d\varphi_s)_{\ga(\eps)}(\dot{\ga}(\eps))$ and $(\ptl F/\ptl s)(\eps,s)=T_{F(\eps,s)}$. This shows that $\Sg_\ga$ is a complete vertical surface. The Riemannian unit normal over $\Sg_\ga$ is given by $N(\eps,s)=-(d\varphi_s)_{\ga(\eps)}(J(\dot{\ga}(\eps)))$. So, the characteristic field $Z(\eps,s)$ equals $(\ptl F/\ptl\eps)(\eps,s)$. The fact that $\ga$ is a CC-geodesic of curvature $\la$ implies that $\Sg_\ga$ has constant mean curvature $\la$. This proves (i). To prove (ii) we take a characteristic curve $\ga$ of $\Sg$. By Proposition~\ref{prop:varprop} (iii) we have that $\ga$ is a complete CC-geodesic in $M$.  Note that the integral curves of $T$ passing through $\ga$ are entirely contained in $\Sg$.  As $\Sg$ is complete and connected we conclude that $\Sg=\Sg_\ga$. Statement (iii) comes from (ii). In fact, we get by Lemma~\ref{lem:geofunction} that the curve $\alpha:=\pi\circ\ga$ has constant geodesic curvature $2H$ in $E$ with respect to the unit normal $-(d\pi)(J(\dot{\ga}))=(d\pi)(N)$. Thus equality $\Sg_\ga=\pi^{-1}(\alpha)$ proves the claim. 
\end{proof}

As an application of Proposition~\ref{prop:vert} we deduce the following result.

\begin{corollary}
\label{cor:vertmodel}
Let $\Sg$ be a complete, connected, vertical surface of constant mean curvature $H$ inside the sub-Riemannian hyperbolic $3$-space $\mathbb{M}(-1)$. Then $\Sg$ is either
\begin{itemize}
\item[(i)] a vertical plane $($a cylinder over a geodesic in $\mathbb{N}(-1)$$)$ if $H=0$, or
\item[(ii)] a hypercylinder $($a cylinder over a hypercycle in $\mathbb{N}(-1)$$)$ if $0<|H|<1$, or
\item[(iii)] a horocylinder $($a cylinder over a horocycle in $\mathbb{N}(-1)$$)$ if $|H|=1$, or
\item[(iv)] a right circular cylinder if $|H|>1$.
\end{itemize}
\end{corollary}

\section{A second variation formula for cmc surfaces in Sasakian $3$-manifolds. 
\\ Stable surfaces}
\label{sec:2ndvar}

Here we study analytical properties of second order minima of the area functional \eqref{eq:area} with or without a volume constraint. 

Let $M$ be a Sasakian sub-Riemannian $3$-manifold. An area-stationary surface $\Sg$ in $M$ is \emph{stable} if $A''(0)\geq 0$ for any variation.  A volume-preserving area-stationary surface $\Sg$ in $M$ is \emph{stable under a volume constraint} if $A''(0)\geq 0$ for any volume-preserving variation.

\begin{remark}
For an area-stationary surface to be stable is stronger than the condition of being stable under a volume constraint. The classification of complete $C^2$ stable area-stationary surfaces in the first Heisenberg group was obtained in \cite{hrr}. In Section~\ref{sec:main} we will analyze which area-stationary surfaces with empty singular set in the Heisenberg group are stable under a volume constraint.
\end{remark}

\subsection{Second variation formula}
\label{subsec:2ndvar}
In this part of the section we take a surface with constant mean curvature $H$ inside a Sasakian sub-Riemannian $3$-manifold and we give an expression for the second
derivative of the functional $A+2HV$ associated to certain variations supported in the regular set. 

The second variation of the area functional \eqref{eq:area} has appeared in several
contexts, see \cite[Prop.~6.1]{chmy}, \cite[Sect.~3.2]{bscv}, \cite{selby}, 
\cite[Sect.~14]{dgn}, \cite[Thm.~4.8]{montefalcone}, \cite[Proof of Thm.~3.5]{mscv}, 
\cite[Thm.~3.7]{hrr}, \cite[Thm.~E]{hp2} and \cite{galli}. A brief description of the obtained  formulas can be found in Section~\ref{sec:intro}. Here we prove the following result.

\begin{theorem}
\label{th:2ndvar}
Let $\Sg$ be a $C^2$ orientable surface immersed in a Sasakian sub-Riemannian $3$-mani\-fold $M$. Let $I$ be an open interval containing the origin and  $\omega:I\times\Sg\to\rr$ a $C^1$ function satisfying  \eqref{eq:omega1}-\eqref{eq:omega4}. Consider a variation of $\Sg$ of the type 
\[
\varphi_{s}(p):=\exp_{p}(\om(s,p)N_{p}),
\]
defined for $p\in\Sg$ and $s$ small enough. Denote by $U$ the associated velocity vector field given by $U_{p}:=u(p)N_{p}$, where $u(p):=(\ptl\omega/\ptl s)(0,p)$.  

If $\Sg$ has constant mean curvature $H$ and the variation $\{\varphi_s\}_s$ is supported off of the singular set $\Sg_0$, then the functional $A+2HV$ is twice differentiable at $s=0$, and
\begin{equation}
\label{eq:gen2ndvar}
(A+2HV)''(0)=\int_{\Sg}|N_{h}|^{-1}\left\{Z(u)^2-
\big(|B(Z)+S|^2+4(K-1)|N_{h}|^2\big)\,u^2\right\}d\Sg,
\end{equation}
where $\{Z,S\}$ is the orthonormal basis defined in \eqref{eq:nuh} and \eqref{eq:ese},
$B$ is the Riemannian shape operator of $\Sg$, and $K$ is the Webster scalar curvature of $M$.
\end{theorem}

\begin{remark}
Formula \eqref{eq:gen2ndvar} is an extension of the second derivative of the area proved in \cite[Thm.~3.7]{hrr} for variations $\varphi_s(p)=\exp_p(s\,u(p)N_p)$ of minimal surfaces in the first Heisenberg group. The proof of \eqref{eq:gen2ndvar} allows also to  get an expression of $(A+2HV)''(0)$ for variations of the type $\exp_p(\om(s,p)N_p+v(p)T_p)$ supported on a regular portion with non-empty boundary of $\Sg$.
\end{remark}

\begin{proof}[Proof of Theorem~\ref{th:2ndvar}]  
We follow closely the proof of \cite[Thm.~3.7]{hrr}.  Let
$\Sg_{s}:=\varphi_{s}(\Sg)$.  We extend the vector field $U$ along the
variation by $U(\varphi_{s}(p))=(d/dt)|_{t=s}\,\varphi_{t}(p)$.  Let $N$ be a
continuous field along the variation whose restriction to any
$\Sg_{s}$ is a unit normal vector.  

We first compute $A''(0)$.  By taking into
account \eqref{eq:area}, the coarea formula, and that the Riemannian
area of $\Sg_{0}$ vanishes, we have
\begin{equation}
\label{eq:ameno}
A(s)=A(\Sg_{s})=\int_{\Sg_{s}}\mnh\,d\Sg_{s}=
\int_{\Sg-\Sg_{0}}\big(\mnh\circ\var_{s}\big) \,|\text{Jac}\,\var_{s}|\,d\Sg,
\end{equation}
where $\text{Jac}\,\var_{s}$ is the Jacobian determinant of the
diffeomorphism $\var_{s}:\Sg\to\Sg_{s}$.  Clearly
$|\text{Jac}\,\var_{0}|=1$ since $\var_{0}(p)=p$ for any $p\in\Sg$. As the variation has compact support in $\Sg-\Sg_0$ we can find, by continuity, an interval $I'=(-s_0,s_0)$ such that $N_h(\varphi_s(p))\neq 0$ whenever $p\in\Sg-\Sg_0$ and $s\in I'$.

Fix a point $p\in\Sg-\Sg_{0}$ and consider the orthonormal basis
$\{e_1,e_{2}\}$ of $T_{p}\Sg$ given by $e_{1}=Z_{p}$ and
$e_{2}=S_{p}$.  Let $\ga_p(s):=\var_{s}(p)=c_p(\om(s,p))$, where $c_p$ is the geodesic in $(M,g)$ with $c_p(0)=0$ and $\dot{c}_p(0)=N_p$.  Let
$\alpha_{i}:(-\eps_{0},\eps_{0})\to\Sg-\Sg_{0}$ be a $C^1$ curve such
that $\alpha_{i}(0)=p$ and $\dot{\alpha}_{i}(0)=e_{i}$. Define the $C^1$ map 
$F_i(\eps,s):=c_{\alpha_{i}(\eps)}(s)$. It can be proved as in Lemma~\ref{lem:ccjacobi} that $(\ptl F_i/\ptl\eps)(0,s)$ is a $C^\infty$ vector field along $c_p(s)$. Consider the $C^1$ map $G_{i}:(-\eps_{0},\eps_{0})\times I'\to M$ given by
$G_{i}(\eps,s):=\var_{s}(\alpha_{i}(\eps))=F_i(\eps,\omega(s,\alpha_{i}(\eps)))$. If we denote $E_{i}(s):=(\ptl G_{i}/\ptl\eps)(0,s)=e_{i}(\var_{s})$ then we have $E_i(0)=e_i$, and
\[
E_i(s)=\frac{\ptl F_i}{\ptl\eps}(0,\omega(s,p))+
\escpr{\frac{\ptl\om}{\ptl p}(s,p),e_i}\,\dot{c}_p(\omega(s,p)).
\]
In particular, we deduce from \eqref{eq:omega2} and \eqref{eq:omega3} that $E_i$ is $C^2$ along $\ga_p$. Now, we can reproduce the arguments in the proof of Lemma~\ref{lem:ccjacobi} to get that $[\dot{\ga}_p,E_{i}]=0$, the covariant derivatives $D_{E_{i}}D_{U}U$, $D_{U}D_{E_{i}}U$ exist, and $R(U,E_{i})U=D_{E_{i}}D_{U}U-D_{U}D_{E_{i}}U$.
Therefore, we have the following identities
\begin{align}
\label{eq:bracket}
D_{U}E_{i}&=D_{E_{i}}U,
\\
\label{eq:jacobi}
D_{U}D_{U}E_{i}&=-R(U,E_{i})U+D_{E_{i}}D_{U}U,    
\end{align}
valid along $\ga_p$. On the other hand, it is clear that $\{E_{1}(s),E_{2}(s)\} $ is a
basis of the tangent space to $\Sg_{s}$ at $\ga_p(s)$. Hence
$|\text{Jac}\,\var_{s}|=(|E_{1}|^2\,|E_{2}|^2-\escpr{E_{1},E_{2}}^2)^{1/2}(s)$, which is a $C^2$ function along $\ga_p$. Moreover, if
$N(s)$ denotes the unit normal to $\Sg_{s}$ at $\ga_p(s)$, then
$N(s)=\pm|E_{1}\times E_{2}|^{-1}\,(E_{1}\times E_{2})(s)$, where the cross product is taken with respect to a local orthonormal frame of $M$. Thus $N(s)$ is
$C^2$ on $\ga_p$.  We conclude that $|N_{h}|(s)$ is $C^2$ along $\ga_p$ as well. As the support of the variation is contained in $\Sg-\Sg_0$ we can differentiate under the integral sign in \eqref{eq:ameno}. We use primes $'$ for the derivatives with respect to $s$. It follows that
\begin{equation}
\label{eq:2nd1}
A''(0)=\int_{\Sg-\Sg_{0}}\left\{\mnh''(0)
+2\,\mnh'(0)\,|\text{Jac}\,\var_{s}|'(0)
+\mnh\,|\text{Jac}\,\var_{s}|''(0)\right\}d\Sg.
\end{equation}

Now we compute the different terms in \eqref{eq:2nd1}. Let $W=D_UU$ be the acceleration vector field of the variation.  We know that $W=wN$ on $\Sg$, where $w(p):=(\ptl^2\om/\ptl s^2)(0,p)$. The calculus
of $|\text{Jac}\,\var_{s}|'(0)$ and $|\text{Jac}\,\var_{s}|''(0)$ is found
in \cite[Sect.~9]{simon} for $C^2$ variations of a $C^1$ surface in
Euclidean $3$-space.  The arguments can be generalized to the present 
situation. As $U=uN$ on $\Sg$, we deduce
\begin{equation}
\label{eq:jac1}
|\text{Jac}\,\var_{s}|'(0)=\divv_{\Sg}U=(-2H_{R})\,u,
\end{equation}
whereas
\begin{align}
\label{eq:jac2}
|\text{Jac}\,\var_{s}|''(0)=&\divv_{\Sg}W+(\divv_{\Sg}U)^2
+\sum_{i=1}^2|(D_{e_{i}}U)^\bot|^2
\\
\nonumber
&
-\sum_{i=1}^2\escpr{R(U,e_{i})U,e_{i}} 
-\sum_{i,j=1}^2\escpr{D_{e_{i}}U,e_{j}}\,\escpr{D_{e_{j}}U,e_{i}}
\\
\nonumber
=&(-2H_R)\,w+ (4H^2_{R})\,u^2+|\nabla_{\Sg}u|^2- (\ric(N,N)+|B|^2)\,u^2.
\end{align}
In the previous equations  $-2H_{R}=\divv_{\Sg}N$ is the Riemannian
mean curvature of $\Sg$, $\nabla_{\Sg}u$ is the gradient relative to
$\Sg$ of $u$, $\ric$ is the Ricci tensor in $(M,g)$, and $|B|^2$ is
the squared norm of the Riemannian shape operator of $\Sg$. From \eqref{eq:mc2} we infer this relation between $H_{R}$ and~$H$
\begin{equation}
\label{eq:2hr}
2H_{R}=-\divv_{\Sg}N=\escpr{B(Z),Z}+\escpr{B(S),S}=
2H\mnh+\escpr{B(S),S}.
\end{equation}

Let us compute $\mnh'(0)$ and $\mnh''(0)$.  From \eqref{eq:vmnh} and
\eqref{eq:conmute} it follows that
\[
\mnh'(s)=U(\mnh)=\escpr{D_{U}N,\nuh}
+\escpr{N,T}\,\escpr{U,Z}.
\]
The second equality in \eqref{eq:relations} together with $U=uN$ and $D_{U}N=-\nabla_{\Sg}u$ on $\Sg$ implies that
\begin{equation}
\label{eq:dmnh}
|N_{h}|'(0)=-\escpr{N,T}\,S(u).
\end{equation}
We also have
\begin{equation}
\label{eq:d2mnh1}
|N_{h}|''(0)=\escpr{D_{U}D_{U}N,\nuh}+\escpr{D_{U}N,D_{U}\nuh}
+\escpr{N,T}\,\escpr{U,D_{U}Z},
\end{equation}
since $\escpr{U,Z}=\escpr{W,Z}=0$ on $\Sg-\Sg_{0}$.  We can compute
$D_{U}\nuh$ from \eqref{eq:dvnuh}.  By using that
$D_{U}N=-\nabla_{\Sg}u$ and $J(U)=(\mnh u)Z$ on $\Sg-\Sg_{0}$, we get
\begin{equation}
\label{eq:d2mnh12}
\escpr{D_{U}N,D_{U}\nuh}=\mnh^{-1}\,Z(u)^2+\escpr{N,T}\,Z(u)\,u.
\end{equation}
It is also easy to check that $D_{U}Z=\left(\mnh^{-1}Z(u)
+\escpr{N,T}\,u\right)\nuh-(\mnh\, u)\,T$. Thus
\begin{equation}
\label{d2mnh13}
\escpr{U,D_{U}Z}=Z(u)\,u.
\end{equation}
It remains to compute $D_{U}D_{U}N$.  Note that $\{E_{1},E_{2},N\}$
is an orthonormal basis of $T_pM$. Hence
\begin{equation}
\label{eq:d2n1}
D_{U}D_{U}N=\sum_{i=1}^2\escpr{D_{U}D_{U}N,E_{i}}\,E_{i}
+\escpr{D_{U}D_{U}N,N}\,N.
\end{equation}
As $\escpr{N,E_{i}}=0$ along $\ga_p$,  we get
\begin{align}
\label{eq:d2n11}
\escpr{D_{U}D_{U}N,E_{i}}&=-2\escpr{D_{U}N,D_{U}E_{i}}
-\escpr{N,D_{U}D_{U}E_{i}}
\\
\nonumber
&=-2\escpr{D_{U}N,D_{e_{i}}U}+\escpr{R(U,E_{i})U,N}-\escpr{D_{e_{i}}W,N}.
\\
\nonumber
&=2u\,\escpr{\nabla_{\Sg}u,D_{e_{i}}N}-e_i(w).
\end{align}
The second equality follows from \eqref{eq:bracket} and \eqref{eq:jacobi}. For the third one it suffices to take into account $D_{e_{i}}(fN)=e_{i}(f)N+fD_{e_{i}}N$ and $\escpr{R(N,E_{i})N,N}=0$.  Moreover, since $|N|^2=1$ on $\Sg$, we deduce
\begin{equation}
\label{eq:d2n12}
\escpr{D_{U}D_{U}N,N}=-|D_{U}N|^2=-|\nabla_{\Sg}u|^2.
\end{equation}
By substituting \eqref{eq:d2n11} and \eqref{eq:d2n12} into 
\eqref{eq:d2n1}, we obtain
\[
D_{U}D_{U}N=2u\,\sum_{i=1}^2\escpr{\nabla_{\Sg}u,D_{e_{i}}N}\,e_{i}
-\nabla_\Sg w-|\nabla_{\Sg}u|^2\,N.
\]
Now recall that $e_{1}=Z_{p}$ and $e_{2}=S_{p}$.  Then equalities
$D_{S}N=-\escpr{B(Z),S}Z-\escpr{B(S),S}S$ and
$\escpr{S,\nuh}=\escpr{N,T}$ lead us to the following expression
\begin{align}
\label{eq:d2mnh11}
\escpr{D_{U}D_{U}N,\nuh}=&-2\escpr{N,T}\,\escpr{B(Z),S}\,Z(u)\,u
-2\escpr{N,T}\,\escpr{B(S),S}\,S(u)\,u
\\
\nonumber
&-\escpr{N,T}\,S(w)-\mnh\,|\nabla_{\Sg}u|^2.
\end{align}
Equations \eqref{eq:d2mnh11}, \eqref{eq:d2mnh12} and
\eqref{d2mnh13} allow us to compute $\mnh''(0)$ from \eqref{eq:d2mnh1}. By using the resulting formula together with \eqref{eq:dmnh},
\eqref{eq:jac1}, \eqref{eq:jac2} and \eqref{eq:2hr}, we get 
\begin{align}
\label{eq:defini1}
&\mnh''(0)+2\,\mnh'(0)\,|\text{Jac}\,\var_{s}|'(0)
+\mnh\,|\text{Jac}\,\var_{s}|''(0)
\\
\nonumber
&=\mnh^{-1}\,Z(u)^2+2\escpr{N,T}\,Z(u)\,u
-2\escpr{N,T}\,\escpr{B(Z),S}\,Z(u)\,u+4H\mnh\,\escpr{N,T}\,S(u)\,u
\\
\nonumber
&\quad-2H\mnh^2\,w-\mnh\,\escpr{B(S),S}\,w-
\escpr{N,T}\,S(w)+q_{1}\,u^2,
\end{align}
where $q_{1}$ is the function given by
\[
q_{1}=4H^2\,\mnh^3+\mnh\,\escpr{B(S),S}^2+4H\mnh^2\,\escpr{B(S),S}
-\mnh\,(\ric(N,N)+|B|^2).
\]

On the other hand, the derivative $V'(s)$ can be 
computed as in \cite[Lem.~2.1]{bdce}. We have
\[
V'(s)=\int_{\Sg_{s}}\escpr{U,N}\,d\Sg_{s}=
\int_{\Sg}\big{(}\escpr{U,N}\circ\varphi_{s}\big{)}\,|\text{Jac}\,
\var_{s}|\,d\Sg,   
\]
and so 
\begin{equation}
\label{eq:vdosprima}
V''(0)=\int_\Sg\{\escpr{U,N}'(0)+\escpr{U,N}\,|\text{Jac}\,\varphi_s|'(0)\}\,d\Sg.
\end{equation}
By taking into account \eqref{eq:jac1} and
\eqref{eq:2hr}, together with equalities $W=wN$, $U=uN$ and $D_UN=-\nabla_\Sg u$ on $\Sg$,  we deduce 
\begin{align}
\label{eq:defini2}
\escpr{U,N}'(0)+\escpr{U,N}\,|\text{Jac}\,\var_{s}|'(0)=
w-q_{2}\,u^2,
\end{align}
where
\[
q_{2}=2H\mnh+\escpr{B(S),S}. 
\]

We conclude from \eqref{eq:2nd1}, \eqref{eq:defini1}, \eqref{eq:vdosprima} and \eqref{eq:defini2}, that $(A+2HV)''(0)$ is equal to
\begin{align}
\label{eq:nemo}
&\int_{\Sg-\Sg_0}\left\{
\mnh^{-1}\,Z(u)^2+2\escpr{N,T}\,Z(u)\,u
-2\escpr{N,T}\,\escpr{B(Z),S}\,Z(u)\,u\right.
\\
\nonumber
&\left.\qquad\qquad +4H\mnh\,\escpr{N,T}\,S(u)\,u-\escpr{N,T}\,S(w)+aw+
(q_{1}-2Hq_2)\,u^2\right\}d\Sg,
\end{align}
where the function $a$ is given by
\[
a=2H\escpr{N,T}^2-\mnh\,\escpr{B(S),S}.
\]
At this point we apply Lemma~\ref{lem:aux4} below.  After simplifying, the integrand in \eqref{eq:nemo} equals
\begin{align}
\label{eq:largo}
&\mnh^{-1}\,Z(u)^2+\divv_{\Sg}(\xi Z)+\divv_\Sg(\rho S)
+\left\{a+S(\escpr{N,T})+\escpr{N,T}\,q_4\right\}w
\\
\nonumber
&+\bigg\{
q_1-2Hq_2-2H\mnh\,S(\escpr{N,T})+\big(\escpr{B(Z),S}-1\big)\,
\big(\escpr{N,T}\,q_3+Z(\escpr{N,T})\big)
\\
\nonumber
&-\escpr{N,T}\big(2H\mnh\,q_4+2H\,S(\mnh)-Z(\escpr{B(Z),S})
\big)\bigg\}\,u^2,
\end{align}
where $\xi$ and $\rho$ are the functions
\begin{align*}
\xi&=\escpr{N,T}\left(1-\escpr{B(Z),S}\right)u^2,
\\
\rho&=\escpr{N,T}\left(2H\mnh\,u^2-w\right),
\end{align*}
and $\divv_\Sg(\xi Z)$ is understood in the sense of Remark~\ref{re:divv}. On the other hand, the second equality in Lemma~\ref{lem:ruvu} together with formulas \eqref{eq:vnt} and \eqref{eq:vmnh} gives us 
\begin{align*}
\text{Ric}(N,N)&=2\escpr{N,T}^2+(4K-2)\mnh^2,
\\
S(\escpr{N,T})&=\mnh\,\escpr{B(S),S}, \quad S(\mnh)=-\escpr{N,T}\,\escpr{B(S),S},
\\
Z(\escpr{N,T})&=\mnh\,(\escpr{B(Z),S}-1).
\end{align*}
Therefore, after a straightforward calculus in \eqref{eq:largo}, we conclude that
\begin{align*}
(A+2HV)''(0)&=\int_{\Sg-\Sg_0}|N_{h}|^{-1}\left\{Z(u)^2-
\big(|B(Z)+S|^2+4(K-1)|N_{h}|^2\big)\,u^2\right\}d\Sg
\\
&+\int_{\Sg-\Sg_0}\divv_\Sg(\xi Z)\,d\Sg
+\int_{\Sg-\Sg_0}\divv_\Sg(\rho S)\,d\Sg.
\end{align*}
The proof finishes by applying the Riemannian divergence theorem and \cite[Lem.~2.4]{hrr}.
\end{proof}

\begin{lemma}
\label{lem:aux4}
Let $\Sg$ be a $C^2$ orientable surface immersed in a Sasakian sub-Riemannian $3$-manifold $M$. For any $\phi\in C^1(\Sg)$ we have the following equalities in the regular set $\Sg-\Sg_{0}$
\begin{align}
\label{eq:divv1}
\divv_{\Sg}(\phi Z)&=Z(\phi)+q_{3}\,\phi,
\\
\label{eq:divv2}
\divv_{\Sg}(\phi S)&=S(\phi)+q_{4}\,\phi,
\end{align}
where $q_{3}$ and $q_4$ are given by
\begin{align*}
q_{3}&=|N_{h}|^{-1}\,\escpr{N,T}\,(1+\escpr{B(Z),S}),
\\
q_4&=-2H\,\escpr{N,T}.
\end{align*}
Moreover, if $\Sg$ has constant mean curvature $H$ then, in $\Sg-\Sg_0$, 
the functions $\escpr{N,T}$, $\mnh$ and $\escpr{B(Z),S}$ are
$C^\infty$ in the $Z$-direction, and
\begin{equation}
\label{eq:zbzs}
Z(\escpr{B(Z),S})=4|N_{h}|\,\escpr{N,T}\,(1-K-H^2)-
2|N_{h}|^{-1}\escpr{N,T}\,\escpr{B(Z),S}\,(1+\escpr{B(Z),S}).
\end{equation}
\end{lemma}

\begin{proof}
For any $C^1$ vector field $V$ tangent to $\Sg-\Sg_0$ we have 
$\divv_\Sg(\phi V)=V(\phi)+(\divv_\Sg V)\phi$. To obtain \eqref{eq:divv1} and 
\eqref{eq:divv2} it suffices to check that $\divv_\Sg Z=q_3$ and $\divv_\Sg S=q_4$. 
Clearly
\[
\divv_\Sg S=\escpr{D_ZS,Z}+\escpr{D_SS,S}=\escpr{D_ZS,Z}=-\escpr{S,D_ZZ}=
-2H\escpr{N,T},
\]
where we have used $|S|^2=1$ and \eqref{eq:dzz}. We also get
\[
\divv_\Sg Z=\escpr{D_ZZ,Z}+\escpr{D_SZ,S}=\escpr{D_SZ,S},
\]
which gives us $\divv_\Sg Z=q_3$ if we take into account that
\begin{equation}
\label{eq:dsz}
D_SZ=\mnh^{-1}\left(\escpr{B(Z),S}+1-\mnh^2\right)\nuh-\escpr{N,T}\,T.
\end{equation}

Recall that if $\Sg$ is a CMC surface then the unit normal $N$ is $C^\infty$ in the $Z$-direction by Proposition~\ref{prop:varprop} (iv).  This implies that $\escpr{N,T}$
and $\mnh$ are $C^\infty$ in the $Z$-direction.  As a consequence, the
vector fields $\nuh$ and $S$ defined in \eqref{eq:nuh} and
\eqref{eq:ese} satisfy the same property.  It follows that
$\escpr{B(Z),S}$ is $C^\infty$ in the $Z$-direction.  To compute
$Z(\escpr{B(Z),S})$ note that
\begin{equation}
\label{eq:zbzs1}
Z(\escpr{B(Z),S})=
-\escpr{D_{Z}D_{Z}N,S}-\escpr{D_{Z}N,D_{Z}S}=\escpr{N,D_{Z}D_{Z}S}-\escpr{B(Z),D_{Z}S}.
\end{equation}
From \eqref{eq:dsz} we see that
$D_{S}Z$ is $C^\infty$ in the $Z$-direction.  Hence
$[Z,S]=D_ZS-D_SZ$ is also $C^\infty$ in the $Z$-direction and
$D_{Z}[Z,S]=D_{Z}D_{Z}S-D_{Z}D_{S}Z$.  Therefore we deduce
\begin{align}
\label{eq:zbzs2}
\escpr{N,D_ZD_ZS}&=\escpr{N,D_{Z}[Z,S]}+\escpr{N,D_{Z}D_{S}Z}
\\
\nonumber
&=\escpr{N,D_{Z}[Z,S]}+\escpr{N,D_{S}D_{Z}Z}-\escpr{N,R(Z,S)Z}
+\escpr{N,D_{[Z,S]}Z}
\\
\nonumber 
&=\escpr{N,D_{Z}[Z,S]}-\escpr{N,R(Z,S)Z}
+\escpr{N,D_{[Z,S]}Z},
\end{align}
where $R$ is the curvature tensor in $(M,g)$.  In the third
equality we have used \eqref{eq:dzz} to get
$D_{S}D_{Z}Z=(2H)D_{S}\nuh$, which is proportional to $Z$ by
\eqref{eq:dvnuh}.  Observe that $[Z,S]$ is tangent to $\Sg-\Sg_0$ since
$\escpr{[Z,S],N}=\escpr{D_{Z}S,N}-\escpr{D_{S}Z,N}=-\escpr{S,D_{Z}N}+
\escpr{Z,D_{S}N}=0$. As a consequence, we obtain
\[
\escpr{N,D_{Z}[Z,S]}=\escpr{B(Z),[Z,S]}, \qquad \escpr{N,D_{[Z,S]}Z}=-\escpr{D_{[Z,S]}N,Z}=\escpr{B(Z),[Z,S]}.
\]
If we put this information into \eqref{eq:zbzs2}, we infer from \eqref{eq:zbzs1} that
\begin{equation}
\label{eq:zbzs3}
Z(\escpr{B(Z),S})=\escpr{B(Z),D_{Z}S}-2\escpr{B(Z),D_{S}Z}-\escpr{N,R(Z,S)Z}.
\end{equation}
To compute the first term in \eqref{eq:zbzs3} note that
$\escpr{D_{Z}S,Z}=-\escpr{S,D_{Z}Z}=-2H\escpr{N,T}$ by \eqref{eq:dzz}, whereas $\escpr{D_{Z}S,S}=0$.  In particular
\begin{equation}
\label{eq:bzdzs}
\escpr{B(Z),D_{Z}S}=-2H\escpr{N,T}\,\escpr{B(Z),Z}=-4H^2\mnh\,
\escpr{N,T},
\end{equation}
where we have used \eqref{eq:mc2}.  For the second term in
\eqref{eq:zbzs3} we use \eqref{eq:dsz}. We get
\begin{equation}
\label{eq:bzdsz}
\escpr{B(Z),D_{S}Z}=\mnh^{-1}\,\escpr{N,T}\,\escpr{B(Z),S}\,
(1+\escpr{B(Z),S}).
\end{equation}
For the third term in \eqref{eq:zbzs3}, we apply the first equality in Lemma~\ref{lem:ruvu} to conclude that
\begin{equation}
\label{eq:rzszn}
\escpr{N,R(Z,S)Z}=4(K-1)\mnh\,\escpr{N,T}.
\end{equation}
The proof finishes by substituting \eqref{eq:bzdzs}, \eqref{eq:bzdsz} and
\eqref{eq:rzszn} into \eqref{eq:zbzs3}.
\end{proof}

\subsection{The index form and a stability criterion}
Here we obtain two integral inequalities for stable surfaces with or without a volume constraint. They will play an important role in Section~\ref{sec:main}.

Let $\Sg$ be a $C^2$ orientable CMC surface immersed in a Sasakian sub-Riemannian $3$-manifold $M$.  For any functions
$u,v\in C_{0}(\Sg-\Sg_{0})$ which are also $C^1$ in the $Z$-direction,
we denote
\begin{equation}
\label{eq:indexform}
\mathcal{Q}(u,v):=\int_{\Sg}\mnh^{-1}\left\{Z(u)\,Z(v)-
\big(|B(Z)+S|^2+4(K-1)\mnh^2\big)\,uv\right\}d\Sg,
\end{equation}
where $\{Z,S\}$ is the orthonormal basis in \eqref{eq:nuh} and
\eqref{eq:ese}, $B$ is the Riemannian shape operator of $\Sg$, and $K$ is the Webster scalar curvature of $M$.
The expression \eqref{eq:indexform} defines a symmetric bilinear form.
We refer to $\mathcal{Q}$ as the \emph{index form} associated to $\Sg$
by analogy with the Riemannian situation studied in \cite{bdce}. Observe that $\mathcal{Q}$  contains horizontal analytical terms and geometric information related to the extrinsic shape of $\Sg$ and the curvature of $M$. 

Now we can prove this result.

\begin{proposition}
\label{prop:stcond1}
Let $\Sg$ be a $C^2$ orientable surface immersed in a 
Sasakian sub-Riemannian $3$-manifold $M$. 
\begin{itemize}
\item[(i)] If $\Sg$ is stable, then the index form in \eqref{eq:indexform} satisfies 
$\mathcal{Q}(u,u)\geq 0$ for any  $u\in C_{0}(\Sg-\Sg_{0})$ which is also $C^1$ in the direction of the characteristic field $Z$.
\item[(ii)] If $\Sg$ is stable under a volume constraint, then the index form in \eqref{eq:indexform} satisfies $\mathcal{Q}(u,u)\geq 0$ for any
mean zero function $u\in C_{0}(\Sg-\Sg_{0})$ which is also $C^1$ in
the direction of $Z$.
\end{itemize}
\end{proposition}

\begin{proof}
We only prove (ii). Note that $\Sg$ has constant mean curvature by Proposition~\ref{prop:varprop} (i).  Take $u\in C^1_{0}(\Sg-\Sg_{0})$ with mean zero. By Lemma~\ref{lem:vpvar} we can find a volume-preserving variation as in \eqref{eq:newvar} with $Q_p=N_p$, support contained in $\Sg-\Sg_0$, and associated velocity vector field $U=uN$. We apply Theorem~\ref{th:2ndvar} to deduce that the second derivative of the area functional is
$A''(0)=\mathcal{Q}(u,u)$.  By the stability of $\Sg$ it follows 
that $\mathcal{Q}(u,u)\geq 0$. We have proved that
\begin{equation}
\label{eq:c1}
\mathcal{Q}(u,u)\geq 0,\quad\text{ for any mean zero function } u\in
C^1_{0}(\Sg-\Sg_{0}).
\end{equation}

Now, take a mean zero function $u\in C_{0}(\Sg-\Sg_{0})$ which is also
$C^1$ in the $Z$-direction.  By using \cite[Lem.~2.3]{hrr} and that
$\Sg_{0}$ has vanishing Riemannian area, we can find a compact set
$C\subseteq\Sg-\Sg_{0}$ and a sequence of functions
$\{v_{\eps}\}_{\eps>0}$ in $C^1_{0}(\Sg-\Sg_{0})$ such that
$\{v_{\eps}\}\to u$ in $L^2(\Sg)$, $\{Z(v_{\eps})\}\to Z(u)$ in
$L^2(\Sg)$, while the supports of $v_{\eps}$ and $u$ are contained in
$C$ for any $\eps>0$. Define $u_{\eps}=v_{\eps}+a_{\eps}\mu$, where 
$a_{\eps}=\int_{\Sg}v_{\eps}\,d\Sg$ and $\mu\in C^1_{0}(\Sg)$ has support 
contained in $C$ and integral equal to $-1$. This
provides a sequence of mean zero $C^1$ functions with compact support
contained in $C$  such that $\{u_{\eps}\}\to u$ in $L^2(\Sg)$ and $\{Z(u_{\eps})\}\to Z(u)$ in $L^2(\Sg)$. From here it is not difficult to see that $\lim_{\eps\to 0}\mathcal{Q}(u_{\eps},u_{\eps})=\mathcal{Q}(u,u)$, so that inequality \eqref{eq:c1} proves that $\mathcal{Q}(u,u)\geq 0$.
\end{proof}

\begin{remark}
From \eqref{eq:indexform} we see that the integral inequality $\mathcal{Q}(u,u)\geq 0$ is more restrictive when the ambient manifold has non-negative Webster scalar curvature. In Section~\ref{sec:main} we shall obtain optimal classification results of stable surfaces in such a situation.
\end{remark}

We finish this section with two integration by parts formulas. They are generalizations of those previously obtained in \cite{hrr} for surfaces in the first Heisenberg group.

\begin{proposition}
Let $\Sg$ be a $C^2$ orientable surface inside a Sasakian sub-Riemannian $3$-manifold.  Consider two functions $u\in C_0(\Sg-\Sg_0)$ and $v\in C(\Sg-\Sg_0)$ which are $C^1$ and $C^2$ in the $Z$-direction, respectively. Then we have
\begin{align}
\label{eq:ibp1}
\mathcal{Q}(u,v)&=-\int_{\Sg}u\,\mathcal{L}(v)\,d\Sg,
\\
\label{eq:ibp2}
\int_\Sg\mnh\,Z(u)\,Z(v)\,d\Sg&=-\int_\Sg\mnh\left\{u\,Z(Z(v))+
2\mnh^{-1}\escpr{N,T}\,u\,Z(v)
\right\}d\Sg,
\end{align}
where $\mathcal{L}$ is the second order differential operator defined
by
\begin{align}
\label{eq:lu} 
\mathcal{L}(v):=|N_{h}|^{-1}\big\{Z(Z(v))
&+2\,|N_{h}|^{-1}\,\escpr{N,T}\,\escpr{B(Z),S}\,Z(v)
\\
\nonumber
&+(|B(Z)+S|^2+4(K-1)|N_{h}|^2)\,v\big\}.
\end{align}
\end{proposition}

\begin{proof}
Both formulas can be proved as in  \cite[Prop.~3.14 and Lem.~3.17]{hrr} by using previous computations and the generalized divergence theorem in \cite[Lem.~2.4]{hrr}.
\end{proof}

\section{Classification results for stable surfaces with empty singular set}
\label{sec:main}
\setcounter{equation}{0}

In this section we prove the main results of the paper. They concern stable surfaces with empty singular set inside Sasakian sub-Riemannian $3$-manifolds, with special interest in the model spaces $\e$ introduced in Section~\ref{subsec:sf}.

\subsection{Existence of stable surfaces and curvature}
\label{subsec:main1}
Here we obtain some immediate consequences of Proposition~\ref{prop:stcond1} showing  the relevant role of the Webster scalar curvature in relation to the existence of stable surfaces.

In \cite[Ex.~2 in Sect.~7]{chmy} it is observed that the vertical Clifford torus $\{(z,w)\in\mathbb{S}^3;\,|z|^2=|w|^2\}$ is an unstable area-stationary surface of the sub-Riemannian $3$-sphere. The next result generalizes this fact to a more general situation.

\begin{proposition}
\label{prop:unexistence}
Let $\Sg$ be a $C^2$ orientable, compact, area-stationary surface immersed inside a Sasakian sub-Riemannian $3$-manifold $M$. If $\Sg_0=\emptyset$ and the Webster scalar curvature of $M$ satisfies $K\geq 1$ over $\Sg$, then $\Sg$ is unstable.
\end{proposition}

\begin{proof}
Suppose by contradiction that $\Sg$ is stable. Then Proposition~\ref{prop:stcond1} (i) implies that the index form defined in \eqref{eq:indexform} satisfies $\mathcal{Q}(u,u)\geq 0$ for the function $u=1$. As a consequence, we get
\[
-\int_\Sg\mnh^{-1}\left\{|B(Z)+S|^2+4(K-1)\mnh^2\right\}d\Sg\geq 0.
\]
In particular, we deduce that $B(Z)=-S$ over $\Sg$. Let us see that this is not possible.
Take a characteristic curve $\ga$ of $\Sg$. As $\Sg$ is compact, the curve $\ga$ is well-defined on the whole real line. Let $f:\rr\to\rr$ be the function $f(s):=\escpr{N,T}(\ga(s))$. By taking into account \eqref{eq:vnt} and \eqref{eq:vmnh}, we have
\begin{align*}
f'(s)&=Z(\escpr{N,T})(\ga(s))=-2\mnh(\ga(s))<0,
\\
f''(s)&=-2\,Z(\mnh)(\ga(s))=-4f(s).
\end{align*}
It follows that $f(s)=a\cos(2s)+b\sin(2s)$ which contradicts that $f'(s)<0$ for any $s\in\rr$.
\end{proof}

For the vertical surfaces introduced in Section~\ref{subsec:vertical} we can improve the previous proposition and establish a stability result.

\begin{proposition}
\label{prop:stvert}
Let $\Sg$ be a $C^2$ vertical surface of constant mean curvature $H$ inside a Sasakian sub-Riemannian $3$-manifold $M$ with Webster scalar curvature $K$.
\begin{itemize}
\item[(i)] If $\Sg$ is compact, minimal, and $K>0$ over $\Sg$, then $\Sg$ is unstable.
\item[(ii)] If $H^2+K\leq 0$ over $\Sg$ then the index form defined in \eqref{eq:indexform} satisfies  $\mathcal{Q}(u,u)\geq 0$ for any $u\in C_0(\Sg)$ which is also $C^1$ in the direction of the characteristic field $Z$.
\end{itemize}
\end{proposition} 

\begin{proof}
As $\Sg$ is vertical we have $\escpr{N,T}=0$ and $\mnh=1$ over $\Sg$. By using \eqref{eq:vnt} we obtain $Z(\escpr{N,T})=\mnh(\escpr{B(Z),S}-1)$, so that $\escpr{B(Z),S}=1$ over $\Sg$. By \eqref{eq:mc2} it follows that $B(Z)=(2H)Z+S$. Therefore we get
\[
\mathcal{Q}(u,u)=\int_\Sg\left(Z(u)^2-4(H^2+K)\,u^2\right)d\Sg,
\]
for any $u\in C_0(\Sg)$ which is also $C^1$ in the $Z$-direction. To deduce (i) it is enough to apply Proposition~\ref{prop:stcond1} (i) with $u=1$. The proof of (ii) is immediate.
\end{proof}

\begin{remark}
The second statement in Proposition~\ref{prop:stvert} might suggest that any CMC vertical surface $\Sg$ satisfying $H^2+K\leq 0$ is stable. In precise terms we have proved that  $\Sg$ is stable \emph{under variations as in \eqref{eq:newvar} with $Q_p=N_p$}. 
In Theorem~\ref{th:main} we will show that the condition $H^2+K\leq 0$ is necessary for any CMC surface with empty singular set inside a model space $\e$ to be stable under a volume constraint.
\end{remark}

\subsection{Stable surfaces under a volume constraint in $\e$}
\label{subsec:main2}
Here we classify stable surfaces under a volume constraint with empty singular set inside the sub-Riemannian $3$-space forms $\e$. We will show that the situation is similar to the Riemannian case. In fact, we will be able to characterize exactly such surfaces provided $\kappa\geq 0$, whereas for $\kappa<0$ the complete description will remain open.

The study of the stability condition in $\e$ has focused on stable area stationary surfaces in the first Heisenberg group $\mathbb{M}(0)$, see Section~\ref{sec:intro} for details and references. In \cite[Thm.~4.7]{hrr} and \cite[Thm.~C]{dgnp-stable} it is proved that a complete $C^2$ stable area-stationary surface $\Sg$ with empty singular set in $\mathbb{M}(0)$ must be a vertical plane. The proof follows from the stability inequality in Proposition~\ref{prop:stcond1} (i) by using a suitable cut-off of the function $\mnh$. This function is geometrically associated to the variation of $\Sg$ by equidistant surfaces with respect to the Carnot-Carath\'eodory distance. The aforementioned result leaves open the characterization of complete stable surfaces \emph{under a volume constraint} in $\mathbb{M}(0)$. Inspired by the previous argument, we treat this question by inserting inside the stability inequality in Proposition~\ref{prop:stcond1} (ii) a \emph{mean zero} cut-off of the function $\mnh$.

We first compute the value of the index form for a function $u=f\mnh$.

\begin{lemma}
\label{lem:fmnh}
Let $\Sg$ be a $C^2$ orientable CMC surface immersed inside a Sasakian sub-Riemannian $3$-manifold $M$. For any
function $f\in C_{0}(\Sg-\Sg_{0})$ which is also $C^1$ in the
$Z$-direction, we have
\begin{equation}
\label{eq:indexform3}
\mathcal{Q}(f\mnh,f\mnh)=\int_{\Sg}\mnh\left\{Z(f)^2-
\mathcal{L}(|N_{h}|)\,f^2\right\}d\Sg,
\end{equation}
where $\mathcal{Q}$ is the index form defined in \eqref{eq:indexform},
and $\mathcal{L}$ is the differential operator in
\eqref{eq:lu}.

Moreover, suppose that we are given $u\in C_{0}(\Sg-\Sg_{0})$ and $v\in
C(\Sg-\Sg_{0})$ which are $C^1$ and $C^2$ in the $Z$-direction,
respectively.  If $v$ never vanishes, then we get
\begin{align}
\label{eq:indexform4b}
\mathcal{Q}(uv^{-2}\mnh,uv^{-2}\mnh)&=\int_{\Sg}\mnh\,
v^{-4}\,Z(u)^2\,d\Sg-\int_{\Sg}\mnh\,\mathcal{L}(\mnh)\,(uv^{-2})^2\,d\Sg
\\
\nonumber
&+\int_{\Sg}\mnh\,\bigg\{Z(v^{-2})^2-\frac{1}{2}\,Z(Z(v^{-4}))
-\frac{\escpr{N,T}}{\mnh}\,Z(v^{-4})
\,\bigg\}\,u^2\,d\Sg.
\end{align}
\end{lemma}

\begin{proof}
Note that $\mathcal{L}(\mnh)$ is well-defined since $\mnh$ is $C^\infty$ in the
$Z$-direction by Lemma~\ref{lem:aux4}. To obtain \eqref{eq:indexform3} it suffices to follow the proof of \cite[Lem.~4.1]{hrr} and take into account formula \eqref{eq:ibp1}. Equation \eqref{eq:indexform4b} comes from \eqref{eq:indexform3} with the help of \eqref{eq:ibp2}, see the proof of \cite[Lem.~4.3]{hrr} for details.
\end{proof}
  
Equality \eqref{eq:indexform3} indicates that, for a function $u=f\mnh$ with mean
zero and compact support off of the singular set, the stability inequality $\mathcal{Q}(u,u)\geq 0$ in Proposition~\ref{prop:stcond1} (ii) is more restrictive provided $\mathcal{L}(\mnh)\geq 0$.  By this reason it is interesting to compute $\mathcal{L}(\mnh)$ and to discuss its sign.

\begin{lemma}
\label{lem:modnh}
Let $\Sg$ be a $C^2$ orientable surface immersed in a Sasakian sub-Riemannian $3$-manifold $M$.  If $\Sg$ has constant mean curvature $H$ then, in the regular set $\Sg-\Sg_{0}$, we have
\begin{equation}
\label{eq:lnh}
\mathcal{L}(|N_{h}|)=4\big(|N_{h}|^{-2}\,\escpr{B(Z),S}+H^2+K-1\big),
\end{equation}
where $\mathcal{L}$ is the second order operator in \eqref{eq:lu}, $\{Z,S\}$ is the basis defined in \eqref{eq:nuh} and \eqref{eq:ese}, $B$ is the Riemannian shape operator of $\Sg$, and $K$ is the Webster scalar curvature of $M$.
\end{lemma}

\begin{proof}
By Lemma~\ref{lem:aux4} we deduce that $\mnh$ is $C^\infty$ in the $Z$-direction, so that $\mathcal{L}(\mnh)$ is well-defined.  We compute $Z(|N_{h}|)$ and $Z(Z(|N_{h}|))$.  By \eqref{eq:vmnh} and \eqref{eq:relations} we have
\begin{equation}
\label{eq:znh}
Z(|N_{h}|)=\escpr{N,T}\,\big(1-\escpr{B(Z),S}\big),
\end{equation}
and so
\[
Z(Z(|N_{h}|))=Z(\escpr{N,T})\,\big(1-\escpr{B(Z),S}\big)
-\escpr{N,T}\,Z(\escpr{B(Z),S}).
\]
Now we use \eqref{eq:vnt} and \eqref{eq:zbzs}. 
After a straightforward calculus, we get
\begin{align}
\label{eq:zznh}
Z(Z(|N_{h}|))&=\left(4(H^2+K)-5\right)|N_{h}|-4(H^2+K-1)\,|N_{h}|^3
\\
\nonumber
&+2|N_{h}|^{-1}\,\escpr{B(Z),S}
+2|N_{h}|^{-1}\,\escpr{B(Z),S}^2-3|N_{h}|\,\escpr{B(Z),S}^2.
\end{align}
By substituting \eqref{eq:znh} and \eqref{eq:zznh} into
\eqref{eq:lu}, we obtain
\begin{align*}
\mathcal{L}(|N_{h}|)&=4(H^2+K)-5-4(H^2+K-1)\,\mnh^2
\\
&+2\mnh^{-2}\,\escpr{B(Z),S}+2\mnh^{-2}\,\escpr{B(Z),S}^2
-3\escpr{B(Z),S}^2
\\
&+2\mnh^{-2}\,\escpr{N,T}^2\,\escpr{B(Z),S}\,\big(1-\escpr{B(Z),S}\big)
+|B(Z)+S|^2+4(K-1)\,\mnh^2.
\end{align*}
Finally note that \eqref{eq:mc2} implies
$|B(Z)+S|^2=4H^2\mnh^2+(1+\escpr{B(Z),S})^2$. The claim follows by using this equality into the previous equation, and simplifying.
\end{proof}

Now we are ready to discuss the sign of $\mathcal{L}(\mnh)$. Observe that if $\Sg$ is a CMC vertical surface inside a Sasakian sub-Riemannian $3$-manifold then $\mathcal{L}(\mnh)=4(H^2+K)$ by \eqref{eq:lnh} and the computations in the proof of Proposition~\ref{prop:stvert}. Hence $\mathcal{L}(\mnh)\geq 0$ if and only if $H^2+K\geq 0$, with equality if and only if $H^2+K=0$. In the next proposition we obtain a similar comparison for any CMC surface $\Sg$ with empty singular set in $\e$. The result suggests, somehow, that if $H^2+K\geq 0$ then $\mathcal{L}(\mnh)$ measures how far is $\Sg$ from being a vertical surface. The proof is based on the analysis of the CC-Jacobi field associated to the characteristic curves of $\Sg$.

\begin{proposition}
\label{prop:lnh>0}
Let $M$ be a complete Sasakian sub-Riemannian $3$-manifold with constant Webster scalar curvature $K$. Consider a $C^2$ complete orientable surface $\Sg$ immersed in $M$ with empty singular set and constant mean curvature $H$. If $H^2+K\geq 0$ then the operator $\mathcal{L}$ defined in \eqref{eq:lu} satisfies $\mathcal{L}(\mnh)\geq 0$ on $\Sg$.  Moreover, $\mathcal{L}(\mnh)=0$ over $\Sg$ if and only if $\Sg$ is a vertical surface with $H^2+K=0$.
\end{proposition}

\begin{proof}
Take a point $p\in\Sg$.  Let $I$ be an open interval containing the origin, and 
$\alpha:I\to\Sg$ a piece of the integral curve passing through $p$ of the vector field $S$ defined in \eqref{eq:ese}.  Consider the
characteristic curve $\ga_{\eps}(s)$ of $\Sg$ with $\ga_{\eps}(0)=\alpha(\eps)$.  By
Proposition~\ref{prop:varprop} (iii) any $\ga_{\eps}$ is a piece
of a CC-geodesic in $M$ with curvature $H$.  Let $\ga:=\ga_{0}$.  Define the map
$F:I\times\rr\to\Sg$ given by $F(\eps,s):=\ga_{\eps}(s)$.  Denote
$V(s):=(\ptl F/\ptl\eps)(0,s)$.  We know by Lemma~\ref{lem:ccjacobi} that $V$ is a CC-Jacobi field along $\ga$ such that $[\dot{\ga},V]=0$. Clearly $V(0)=S_p$. We denote by primes $'$ the derivatives of functions depending on $s$, and the covariant derivatives along $\ga$. By equations \eqref{eq:prima1} and \eqref{eq:prima21}, we get
\begin{align}
\label{eq:tukan1}
\escpr{V,T}'&=-2\escpr{V,\nuh},
\\
\nonumber
\escpr{V,T}''&=-2\big(\escpr{V',\nuh}-2H\,\escpr{V,Z}+\escpr{V,T}\big).
\end{align}
By using $[\dot{\ga},V]=0$ and \eqref{eq:dvnuh}, we have
\[
\escpr{V',\nuh}=\escpr{D_ZV,\nuh}=\escpr{D_VZ,\nuh}=-\escpr{Z,D_V\nuh}
=\mnh^{-1}\big(\escpr{B(Z),V}+\escpr{N,T}\,\escpr{V,\nuh}\big),
\]
so that
\begin{equation}
\label{eq:tukan3}
\escpr{V,T}''=-2\mnh^{-1}\,\big(\escpr{B(Z),V}+\escpr{N,T}\,\escpr{V,\nuh}-2H\mnh\,\escpr{V,Z}+\mnh\,\escpr{V,T}\big).
\end{equation}
Now we can see that $\{V,\dot{\ga}\}$ is a basis of the tangent plane to $\Sg$ along $\ga$.  As $\{\dot{\ga},\nuh,T\}$ is an orthonormal basis, it suffices to show that $\escpr{V,T}$ and $\escpr{V,\nuh}$ do not vanish simultaneously.  Suppose that we can find $s_{0}$ such that $\escpr{V,T}(s_{0})=\escpr{V,\nuh}(s_{0})=0$. Then $V(s_{0})$ is proportional to $\dot{\ga}(s_{0})$. Moreover, we infer from \eqref{eq:tukan1}, \eqref{eq:tukan3} and \eqref{eq:mc2} that $\escpr{V,T}'(s_{0})=\escpr{V,T}''(s_{0})=0$.  By taking into account that $\escpr{V,T}$ satisfies the third order equation in Lemma~\ref{lem:ccjacobi} (iv) we deduce that $\escpr{V,T}=0$ along $\ga$, a contradiction since $\escpr{V,T}(0)=-\mnh(p)<0$. This proves the claim and allows to conclude that $\escpr{V,T}$ never vanishes along $\ga$ since $\Sg$ has empty singular set. We distinguish two cases.

\emph{Case 1.} If $H^2+K=0$ then by Lemma~\ref{lem:jacobicon} (ii) we have
\begin{equation}
\label{eq:poli1}
\escpr{V,T}(s)=as^2+bs+c.
\end{equation}
The constants $a,b,c$ can be computed from \eqref{eq:tukan1} and \eqref{eq:tukan3}. They are given by
\begin{align*}
a&=(1/2)\escpr{V,T}''(0)=-\mnh^{-1}(p)\,\big(\escpr{B(Z),S}+\escpr{N,T}^2-\mnh^2\big)(p),
\\
b&=\escpr{V,T}'(0)=-2\escpr{N,T}(p),
\\
c&=\escpr{V,T}(0)=-\mnh(p).
\end{align*}
Hence the fact that $\escpr{V,T}(s)\neq 0$ for any $s\in\rr$ is equivalent to that $b^2-4ac\leq 0$. On the other hand, it follows from \eqref{eq:lnh} and an easy computation, that
\[
b^2-4ac=-\mnh^2(p)\,\mathcal{L}(\mnh)(p),
\]
which implies that $\mathcal{L}(\mnh)(p)\geq 0$ with equality if and only if $a=b=0$, that is, $\escpr{N,T}(p)=0$ and $\escpr{B(Z),S}(p)=1$. In particular, if $\mathcal{L}(\mnh)=0$ over $\Sg$ then $\Sg$ is a vertical surface. Conversely, in any vertical surface we have $\escpr{N,T}=0$ and $\escpr{B(Z),S}=1$, so that $\mathcal{L}(\mnh)=0$.

\emph{Case 2.} If $H^2+K>0$ then Lemma~\ref{lem:jacobicon} (iii) yields
\begin{equation}
\label{eq:poli2}
\escpr{V,T}(s)=\frac{1}{\sqrt{\mu}}\,\big(a\,\sin(\sqrt{\mu}\,s)
-b\,\cos(\sqrt{\mu}\,s)\big)+c,
\end{equation}
where $\mu:=4(H^2+K)$ and the constants $a,b,c$ are given by
\begin{align*}
a&=\escpr{V,T}'(0)=-2\escpr{N,T}(p),
\\
b&=\frac{1}{\sqrt{\mu}}\,\escpr{V,T}''(0)=\frac{-2}{\sqrt{\mu}}\,\mnh^{-1}(p)\,\big(\escpr{B(Z),S}+\escpr{N,T}^2-\mnh^2\big)(p),
\\
c&=\frac{1}{\mu}\,\escpr{V,T}''(0)+\escpr{V,T}(0)=\frac{b}{\sqrt{\mu}}-\mnh(p).
\end{align*}
Define the function $g:\rr\to\rr$ by $g(x)=a\sin(x)-b\cos(x)$.  An elementary analysis shows that
$\escpr{V,T}(s)\neq 0$ for any $s\in\rr$ if and only if $d-b>m$, where $d=\sqrt{\mu}\,\mnh(p)$ and $m=\text{max}\,\{g(x);x\in\rr\}=(a^2+b^2)^{1/2}$. We conclude that $(d-b)^2>m^2=a^2+b^2$, and so
\[
0<d^2-2db-a^2=(\mu-4)\,\mnh^2(p)+4\escpr{B(Z),S}(p)=\mnh^2(p)\,\mathcal{L}(\mnh)(p),
\]
where we have used \eqref{eq:lnh} to obtain the last equality. This shows that $\mathcal{L}(\mnh)(p)>0$.
\end{proof}

Now we can prove the main result of the paper.

\begin{theorem}
\label{th:main}
Let $\Sg$ be a $C^2$ complete orientable surface with empty singular set immersed in $\e$. If $\Sg$ is stable under a volume constraint, then the mean curvature $H$ of $\Sg$ satisfies $H^2+\kappa\leq 0$. Moreover, if $H^2+\kappa=0$ then $\Sg$ is a vertical surface in $\e$.
\end{theorem}

\begin{proof}
To prove the claim it suffices, by virtue of Proposition~\ref{prop:lnh>0}, to see that if $H^2+\kappa\geq 0$ and $\mathcal{L}(\mnh)\neq 0$ over $\Sg$, then $\Sg$ is unstable under a volume constraint.

By Proposition~\ref{prop:lnh>0} we can find $p\in\Sg$ such that $\mathcal{L}(\mnh)(p)>0$. Let $\alpha$ be the integral curve through $p$ of the vector field $S$ defined in \eqref{eq:ese}.  Let $\ga_\eps(s)$ be the characteristic curve of $\Sg$ with $\ga_\eps(0)=\alpha(\eps)$. As $\Sg$ is a complete surface with empty singular set we deduce, by Proposition~\ref{prop:varprop} (iii), that any $\ga_\eps$ is a complete CC-geodesic of $\e$ with curvature $H$. From Lemma~\ref{lem:geomodel} we know that either any $\ga_\eps$ is an injective curve defined on $\rr$, or any $\ga_\eps$ is a closed curve of length $L$. In the first case we denote $I'=\rr$. In the second one $I'$ is a circle of length $L$. We can also take a bounded interval $I$ containing the origin such that the $C^1$ map $F:I\times I'\to\Sg$ given by $F(\eps,s):=\ga_{\eps}(s)$ parameterizes a neighborhood of $\Sg$. Let $V_{\eps}(s):=(\ptl F/\ptl\eps)(\eps,s)$. This is a CC-Jacobi field along $\ga_\eps$ by Lemma~\ref{lem:ccjacobi}. Moreover, the function $\escpr{V_\eps,T}$ never vanishes since $\Sg$ has empty singular set. The fact that $V_\eps(0)=S_{\alpha(\eps)}$ implies that $\escpr{V_{\eps},T}<0$. 
Define the function $f_{\eps}:=\escpr{V_{\eps},S}$. We have $\escpr{V_\eps,T}=-f_\eps\mnh$ and $\escpr{V_\eps,\nuh}=f_\eps\,\escpr{N,T}$, where $\mnh$ and $\escpr{N,T}$ are evaluated along $\ga_\eps$. Hence $f_\eps=-\mnh^{-1}\escpr{V_{\eps},T}$, which is positive along $\ga_{\eps}$.  The Riemannian area element of $\Sg$ with respect
to the coordinates $(\eps,s)$ is given by
\begin{equation}
\label{eq:cepsa}
d\Sg=(|V_{\eps}|^2-\escpr{V_{\eps},\dot{\ga}_{\eps}}^2)^{1/2}=  
f_{\eps}\,d\eps\,ds.
\end{equation}
We define the function 
\begin{equation}
\label{eq:uve}
v(\eps,s):=|\escpr{V_{\eps},T}(s)|^{1/2}=(f_{\eps}\mnh)^{1/2},
\end{equation}
which is clearly positive, continuous on $I\times I'$, and $C^\infty$ with respect to $s$ by Lemma~\ref{lem:aux4}. Set $v_\eps(s):=v(\eps,s)$. If we denote by primes $'$ the derivatives with respect to $s$ then, by using \eqref{eq:tukan1} and
\eqref{eq:tukan3}, we obtain
\begin{align*}
(v^{-2}_\eps)'&=\escpr{V_\eps,T}^{-2}\,\escpr{V_\eps,T}'
=-2\escpr{V_\eps,T}^{-2}\,\escpr{V_\eps,\nuh}=\frac{-2\escpr{N,T}}{f_\eps\mnh^2},
\\
(v_\eps^{-4})'&=-2\escpr{V_\eps,T}^{-3}\,\escpr{V_\eps,T}'=
4\escpr{V_\eps,T}^{-3}\,\escpr{V_\eps,\nuh}=\frac{-4\escpr{N,T}}{f_\eps^2\mnh^3},
\\
(v_\eps^{-4})''&=-2\,\big(\escpr{V_\eps,T}^{-3}\,\escpr{V_\eps,T}'\big)'=
6\escpr{V_\eps,T}^{-4}\,(\escpr{V_\eps,T}')^2-2\escpr{V_\eps,T}^{-3}\escpr{V_\eps,T}''
\\
&=\frac{24\escpr{N,T}^2}{f_\eps^2\mnh^4}-\frac{4}{f_\eps^3\mnh^4}
\,\big(\escpr{B(Z),V_\eps}+\escpr{N,T}\,\escpr{V_\eps,\nuh}-2H\mnh\escpr{V_\eps,Z}
+\mnh\,\escpr{V_\eps,T}\big)
\\
&=\frac{24\escpr{N,T}^2}{f_\eps^2\mnh^4}-\frac{4}{f_\eps^2\mnh^4}
\,\big(\escpr{B(Z),S}+\escpr{N,T}^2-\mnh^2\big),
\end{align*}
where in the last equality we have taken into account $V_\eps=\escpr{V_\eps,Z}Z+f_\eps S$ and \eqref{eq:mc2}. After simplifying, we conclude by \eqref{eq:lnh} that
\begin{align}
\label{eq:cepsa2}
\big((v_\eps^{-2})'\big)^2-\frac{1}{2}\,(v_\eps^{-4})''
&-\frac{\escpr{N,T}}{\mnh}\,(v_\eps^{-4})'=\frac{2\escpr{B(Z),S}-2}{f_\eps^2\mnh^4}
\\
\nonumber
&=\frac{\mathcal{L}(\mnh)}{2f_\eps^2\mnh^2}
-\frac{2\big((H^2+\kappa)\mnh^2+\escpr{N,T}^2\big)}
{f_\eps^2\mnh^4}.
\end{align}

Now we construct a family of functions that we will use to show that $\Sg$ is unstable.  Take $\phi:\rr\to\rr$
such that $\phi\in C_{0}^\infty(I)$, $\phi(0)>0$, and
$\int_{I}\phi(\eps)d\eps=0$.  Let $\rho$ be a positive constant so that
$|\phi'(\eps)|\leq\rho$ for any $\eps\in\rr$. For any $n\in\mathbb{N}$, we
consider the function $u_n:I\times I'\to\rr$ given by
\[
u_{n}(\eps,s):=\phi(\eps)\,\phi_n(s), \quad\text{ where }\quad\phi_n(s):=
\begin{cases}
\phi(s/n),\quad \text{ if  } I'=\rr,
\\
\phi(0), \hspace{0.73cm} \text{ if } I' \text{ is a circle. }
\end{cases}
\]
Finally we define $\overline{u}_{n}:=u_{n}v^{-2}\,\mnh$. Clearly $\overline{u}_n$ is a continuous function with compact support on $I\times I'$; in fact, $\text{supp}(u_n)\subset I\times nI$ if $I'=\rr$. Moreover, $\overline{u}_n$ is $C^\infty$ with respect to $s$ by Lemma~\ref{lem:aux4}. By using the coarea formula together with \eqref{eq:uve}, \eqref{eq:cepsa}, and Fubini's theorem, we obtain
\[
\int_{\Sg}
\overline{u}_{n}\,d\Sg=\int_{I\times I'}u_{n}(\eps,s)\,d\eps\,ds=
\left(\int_{I}\phi(\eps)\,d\eps\right)\,\left(\int_{I'}\phi_n(s)\,ds\right)=0.
\]
If we compute the value of the index form \eqref{eq:indexform} over 
$\overline{u}_{n}$ with respect to the coordinates $(\eps,s)$, then we deduce,
by \eqref{eq:indexform4b}, \eqref{eq:uve}, \eqref{eq:cepsa}, and \eqref{eq:cepsa2}, that
\begin{equation}
\label{eq:indexfinal}
\mathcal{Q}(\overline{u}_{n},\overline{u}_{n})=\int_{I\times I'}\frac{(\ptl 
u_{n}/\ptl s)^2}{v^2}\,d\eps\,ds
-\frac{1}{2}\int_{I\times I'}\frac{\mathcal{L}(\mnh)}{v^2}
\,u_{n} ^2\,d\eps\,ds-r_{n}, 
\end{equation}
where
\[
r_{n}=2\int_{I\times I'}\frac{(H^2+\kappa)\,\mnh^2+\escpr{N,T}^2}{v^2\,\mnh^2}\,u^2_{n} \,d\eps\,ds.  
\]

Suppose that $\Sg$ is stable under a volume constraint. By  
Proposition~\ref{prop:stcond1} (ii) this would imply that 
$\mathcal{Q}(\overline{u}_{n},\overline{u}_{n})\geq 0$ for any 
$n\in\mathbb{N}$. By taking into account \eqref{eq:indexfinal} and 
that $r_{n}\geq 0$, we would get, for any $n\in\mathbb{N}$, that
\[
\frac{1}{2}\int_{I\times I'}\frac{\mathcal{L}(\mnh)}{v^2}
\,u_{n} ^2\,d\eps\,ds\leq \int_{I\times I'}\frac{(\ptl 
u_{n}/\ptl s)^2}{v^2}\,d\eps\,ds.
\]
Note that $\{u_{n}\}_{n\in\mathbb{N}}$ pointwise converges when
$n\to\infty$ to $u(\eps,s):=\phi(0)\,\phi(\eps)$.  By
Proposition~\ref{prop:lnh>0} we have $\mathcal{L}(\mnh)\geq 0$ on
$\Sg$ since we are assuming that $H^2+\kappa\geq 0$.  Thus we 
would be able to apply Fatou's lemma to deduce
\[
\label{eq:crucial}
\frac{1}{2}\int_{I\times I'}\frac{\mathcal{L}(\mnh)}{v^2}
\,u^2\,d\eps\,ds\leq\liminf_{n\to\infty}\int_{I\times I'}\frac{(\ptl 
u_{n}/\ptl s)^2}{v^2}\,d\eps\,ds.
\]
In the previous equation the left-hand term is strictly positive 
since $\mathcal{L}(\mnh)$ and $u$ are continuous positive functions 
at the point $p$. To obtain the desired contradiction it is 
enough to see that the right-hand term equals $0$. This is clear if $I'$ is a circle: in that case $\phi_n$ is constant. On the other hand, for $I'=\rr$ observe that
\[
\int_{I\times I'}\frac{(\ptl u_{n}/\ptl s)^2}{v^2}\,d\eps\,ds=\frac{1}{n^2}\,\int_{I\times nI}\frac{\phi(\eps)^2\,\phi'(s/n)^2}{v(\eps,s)^2}\,d\eps\,ds\leq\frac{\rho^2}{n^2}
\int_{I}\phi(\eps)^2\cdot\left(\int_{nI}\frac{1}{v_\eps(s)^2}\,ds\right)
d\eps.  
\]
Recall that $v_\eps(s)^2=-\escpr{V_{\eps},T}(s)$. Having in mind the expression of $\escpr{V_\eps,T}$ in equations \eqref{eq:poli1} and \eqref{eq:poli2} it is easy to check that $m(\eps):=\max\{v_\eps(s)^{-2};s\in\rr\}$ is a continuous function of $\eps\in I$. If we denote $\ell:=\text{length}(I)$ then we conclude that
\[
0\leq\int_{I\times I'}\frac{(\ptl 
u_{n}/\ptl s)^2}{v^2}\,d\eps\,ds\leq\frac{\rho^2\ell}{n}
\int_{I}\phi(\eps)^2\,m(\eps)\,d\eps,
\]
which tends to $0$ when $n\to\infty$. This completes the proof.
\end{proof}

\begin{remark}
Theorem~\ref{th:main} holds with the same proof for homogeneous Sasakian sub-Riemannian $3$-manifolds where all the complete CC-geodesics of a given curvature are either injective curves defined on $\rr$, or closed curves with the same length. This happens for example in the two models for the sub-Riemannian space $\text{SL}(2,\rr)$ described in \cite{markina} and \cite{boscain-rossi}.
\end{remark}

By using Theorem~\ref{th:main} together with the classification of complete CMC vertical surfaces in $\e$ given in \cite[Prop.~6.16]{rr2} and Corollary~\ref{cor:vertmodel}, we deduce the following result.

\begin{corollary}
\label{cor:main}
Let $\Sg$ be a $C^2$ complete, connected, orientable surface with empty singular set immersed in $\e$. Then we have
\begin{itemize}
\item[(i)] If $\kappa=0$ $($the first Heisenberg group$)$ and $\Sg$ is stable under a volume constraint, then $\Sg$ is a vertical plane.
\item[(ii)] If $\kappa=1$ $($the sub-Riemannian $3$-sphere$)$ then $\Sg$ is unstable under a volume constraint.
\item[(iii)] If $\kappa=-1$ $($the sub-Riemannian hyperbolic $3$-space$)$ and $\Sg$ is stable under a volume constraint, then the mean curvature $H$ of $\Sg$ satisfies $H^2\leq 1$. Moreover, if $H^2=1$ then $\Sg$ is a horocylinder. 
\end{itemize}
\end{corollary}

\begin{remark}
1. As a particular case of Corollary~\ref{cor:main} we deduce that a complete, orientable, minimal surface with empty singular set in the first Heisenberg group $\mathbb{M}(0)$ which is also stable under a volume constraint must be a vertical plane. In the definition of stability given in \cite{hrr} the volume constraint is not included. Therefore the previous corollary is stronger than the classification result in \cite[Thm.~4.7]{hrr} and \cite[Thm.~C]{dgnp-stable}. Note also that in $\mathbb{M}(0)$ the vertical planes are area-minimizing surfaces, see \cite[Ex.~2.2]{bscv}.

2. The isoperimetric problem in the sub-Riemannian three-sphere $\mathbb{M}(1)$ seeks minimizers for the perimeter functional \eqref{eq:per} enclosing a fixed volume. So, if a minimizer of class $C^2$ exists then the boundary is a stable surface under a volume constraint. It follows from Corollary~\ref{cor:main} that any $C^2$ solution to the isoperimetric problem in $\mathbb{M}(1)$ has non-empty singular set. Hence it must be one of the volume-preserving area-stationary surfaces with non-empty singular set described in \cite[Sect.~5]{hr1}. In \cite{hr2} we show that the rotationally invariant CMC spheres of $\mathbb{M}(1)$ are stable surfaces under a volume constraint, so that they provide natural candidates to solve the isoperimetric problem.

3. If $\Sg$ is a CMC vertical surface with $H^2\leq 1$ in the sub-Riemannian hyperbolic $3$-space $\mathbb{M}(-1)$ then we proved in Proposition~\ref{prop:stvert} that $\Sg$ is stable under variations as in \eqref{eq:newvar} with $Q_p=N_p$. However, the stability condition in $\mathbb{M}(-1)$ is not so restrictive as in $\e$ with $\kappa\geq 0$ since the Webster scalar curvature is strictly negative. This leads us to expect the existence in $\mathbb{M}(-1)$ of complete stable surfaces, different from vertical ones, having empty singular set and mean curvature $0\leq H^2<1$. 
\end{remark}

\providecommand{\bysame}{\leavevmode\hbox to3em{\hrulefill}\thinspace}
\providecommand{\MR}{\relax\ifhmode\unskip\space\fi MR }
\providecommand{\MRhref}[2]{%
  \href{http://www.ams.org/mathscinet-getitem?mr=#1}{#2}
}
\providecommand{\href}[2]{#2}


\begin{thebibliography}{10}

\bibitem{balogh}
Z.~M. Balogh, \emph{Size of characteristic sets and functions with prescribed
  gradient}, J. Reine Angew. Math. \textbf{564} (2003), 63--83. \MR{MR2021034
  (2005d:43007)}

\bibitem{bdc}
J.~L. Barbosa and M.~P. do~Carmo, \emph{Stability of hypersurfaces with
  constant mean curvature}, Math. Z. \textbf{185} (1984), no.~3, 339--353.
  \MR{MR731682 (85k:58021c)}

\bibitem{bdce}
J.~L. Barbosa, M.~P. do~Carmo, and J.~Eschenburg, \emph{Stability of
  hypersurfaces of constant mean curvature in {R}iemannian manifolds}, Math. Z.
  \textbf{197} (1988), no.~1, 123--138. \MR{MR917854 (88m:53109)}

\bibitem{bscv}
V.~Barone~Adesi, F.~Serra~Cassano, and D.~Vittone, \emph{The {B}ernstein
  problem for intrinsic graphs in {H}eisenberg groups and calibrations}, Calc.
  Var. Partial Differential Equations \textbf{30} (2007), no.~1, 17--49.
  \MR{MR2333095}

\bibitem{andre}
A.~Bella{\"{\i}}che, \emph{The tangent space in sub-{R}iemannian geometry},
  Sub-Riemannian geometry, Progress in Mathematics, vol. 144, Birkh\"auser,
  Basel, 1996, pp.~1--78. \MR{MR1421822 (98a:53108)}

\bibitem{falbel4}
P.~Bieliavsky, E.~Falbel, and C.~Gorodski, \emph{The classification of
  simply-connected contact sub-{R}iemannian symmetric spaces}, Pacific J. Math.
  \textbf{188} (1999), no.~1, 65--82. \MR{MR1680411 (2000d:53051)}

\bibitem{blair}
D.~E. Blair, \emph{Riemannian geometry of contact and symplectic manifolds},
  Progress in Mathematics, vol. 203, Birkh\"auser Boston Inc., Boston, MA,
  2002. \MR{MR1874240 (2002m:53120)}

\bibitem{boscain-rossi}
U.~Boscain and F.~Rossi, \emph{Invariant {C}arnot-{C}aratheodory metrics on
  {$S^3,\ {\rm SO}(3),\ {\rm SL}(2)$}, and lens spaces}, SIAM J. Control Optim.
  \textbf{47} (2008), no.~4, 1851--1878. \MR{MR2421332 (2009f:53043)}

\bibitem{boyer}
C.~P. Boyer and K.~Galicki, \emph{Sasakian geometry}, Oxford Mathematical
  Monographs, Oxford University Press, Oxford, 2008. \MR{MR2382957
  (2009c:53058)}

\bibitem{cdg1}
L.~Capogna, D.~Danielli, and N.~Garofalo, \emph{The geometric {S}obolev
  embedding for vector fields and the isoperimetric inequality}, Comm. Anal.
  Geom. \textbf{2} (1994), no.~2, 203--215. \MR{MR1312686 (96d:46032)}

\bibitem{survey}
L.~Capogna, D.~Danielli, S.~D. Pauls, and J.~T. Tyson, \emph{An introduction to
  the {H}eisenberg group and the sub-{R}iemannian isoperimetric problem},
  Progress in Mathematics, vol. 259, Birkh\"auser Verlag, Basel, 2007.
  \MR{MR2312336}

\bibitem{markina}
D.-C. Chang, I.~Markina, and A.~Vasil'ev, \emph{Sub-{L}orentzian geometry on
  anti-de {S}itter space}, J. Math. Pures Appl. (9) \textbf{90} (2008), no.~1,
  82--110. \MR{MR2435216 (2009h:53060)}

\bibitem{chanillo-yang}
S.~Chanillo and P.~Yang, \emph{Isoperimetric inequalities \& volume comparison
  theorems on {CR} manifolds}, Ann. Sc. Norm. Super. Pisa Cl. Sci. (5)
  \textbf{8} (2009), no.~2, 279--307. \MR{MR2548248}

\bibitem{chenghwang}
J.-H. Cheng and J.-F. Hwang, \emph{Properly embedded and immersed minimal
  surfaces in the {H}eisenberg group}, Bull. Austral. Math. Soc. \textbf{70}
  (2004), no.~3, 507--520. \MR{MR2103983 (2005f:53010)}

\bibitem{chmy}
J.-H. Cheng, J.-F. Hwang, A.~Malchiodi, and P.~Yang, \emph{Minimal surfaces in
  pseudohermitian geometry}, Ann. Sc. Norm. Super. Pisa Cl. Sci. (5) \textbf{4}
  (2005), no.~1, 129--177. \MR{MR2165405 (2006f:53008)}

\bibitem{chy}
J.-H. Cheng, J.-F. Hwang, and P.~Yang, \emph{Existence and uniqueness for
  {$p$}-area minimizers in the {H}eisenberg group}, Math. Ann. \textbf{337}
  (2007), no.~2, 253--293. \MR{MR2262784}

\bibitem{chern}
S.~S. Chern and R.~S. Hamilton, \emph{On {R}iemannian metrics adapted to
  three-dimensional contact manifolds}, Workshop {B}onn 1984 ({B}onn, 1984),
  Lecture Notes in Math., vol. 1111, Springer, Berlin, 1985, With an appendix
  by Alan Weinstein, pp.~279--308. \MR{MR797427 (87b:53060)}

\bibitem{silveira}
A.~M. Da~Silveira, \emph{Stability of complete noncompact surfaces with
  constant mean curvature}, Math. Ann. \textbf{277} (1987), no.~4, 629--638.
  \MR{MR901709 (88h:53053)}

\bibitem{dgn}
D.~Danielli, N.~Garofalo, and D.-M. Nhieu, \emph{Sub-{R}iemannian calculus on
  hypersurfaces in {C}arnot groups}, Adv. Math. \textbf{215} (2007), no.~1,
  292--378. \MR{MR2354992}

\bibitem{dgnp-stable}
D.~Danielli, N.~Garofalo, D.~M. Nhieu, and S.~D. Pauls, \emph{The {B}ernstein
  problem for embedded surfaces in the {H}eisenberg group $\mathbb{H}^1$},
  preprint. Revision of \emph{Stable complete embedded minimal surfaces in
  $\mathbb{H}^1$ with empty characteristic locus are vertical planes},
  arXiv:0903.4296.

\bibitem{dgnp}
\bysame, \emph{Instability of graphical strips and a positive answer to the
  {B}ernstein problem in the {H}eisenberg group {$\Bbb H^1$}}, J. Differential
  Geom. \textbf{81} (2009), no.~2, 251--295. \MR{MR2472175}

\bibitem{d2}
M.~Derridj, \emph{Sur un th\'eor\`eme de traces}, Ann. Inst. Fourier (Grenoble)
  \textbf{22} (1972), no.~2, 73--83. \MR{MR0343011 (49 \#7755)}

\bibitem{diniz}
M.~M. Diniz, \emph{Variedades sub-riemannianas de contacto de dimens\~ao 3},
  Ph.D. thesis, Universidade de S\~ao Paulo, 1996.

\bibitem{dcriem}
M.~P. do~Carmo, \emph{Riemannian geometry}, Mathematics: Theory \&
  Applications, Birkh\"auser Boston Inc., Boston, MA, 1992, Translated from the
  second Portuguese edition by Francis Flaherty. \MR{MR1138207 (92i:53001)}

\bibitem{falbel3}
E.~Falbel and C.~Gorodski, \emph{On contact sub-{R}iemannian symmetric spaces},
  Ann. Sci. \'Ecole Norm. Sup. (4) \textbf{28} (1995), no.~5, 571--589.
  \MR{MR1341662 (96m:53056)}

\bibitem{falbel1}
\bysame, \emph{Sub-{R}iemannian homogeneous spaces in dimensions {$3$} and
  {$4$}}, Geom. Dedicata \textbf{62} (1996), no.~3, 227--252. \MR{MR1406439
  (97g:53060)}

\bibitem{fc}
D.~Fischer-Colbrie, \emph{On complete minimal surfaces with finite {M}orse
  index in three-manifolds}, Invent. Math. \textbf{82} (1985), no.~1, 121--132.
  \MR{MR808112 (87b:53090)}

\bibitem{galli}
M.~Galli, \emph{First and second variation formulae for the sub-{R}iemannian
  area in three-dimensional pseudo-{H}ermitian manifolds}, preprint, 2010.

\bibitem{gp}
N.~Garofalo and S.~D. Pauls, \emph{{The Bernstein Problem in the Heisenberg
  Group}}, arXiv math.DG/0209065 v2, 2002.

\bibitem{gromov-cc}
M.~Gromov, \emph{Carnot-{C}arath\'eodory spaces seen from within},
  Sub-Riemannian geometry, Progress in Mathematics, vol. 144, Birkh\"auser,
  Basel, 1996, pp.~79--323. \MR{MR1421823 (2000f:53034)}

\bibitem{hass}
J.~Hass, \emph{Complete area minimizing minimal surfaces which are not totally
  geodesic}, Pacific J. Math. \textbf{111} (1984), no.~1, 35--38. \MR{MR732056
  (85i:53011)}

\bibitem{hp2}
R.~K. Hladky and S.~D. Pauls, \emph{{Variation of perimeter measure in
  sub-Riemannian geometry}}, arXiv:math/0702237.

\bibitem{hp1}
\bysame, \emph{Constant mean curvature surfaces in sub-{R}iemannian geometry},
  J. Differential Geom. \textbf{79} (2008), no.~1, 111--139. \MR{MR2401420}

\bibitem{hughen}
K.~Hughen, \emph{The geometry of sub-{R}iemannian three manifolds}, Ph.D.
  thesis, Duke University, 1995.

\bibitem{hrr}
A.~Hurtado, M.~Ritor\'e, and C.~Rosales, \emph{The classification of complete
  stable area-stationary surfaces in the {H}eisenberg group $\mathbb{H}^1$},
  Adv.~Math. \textbf{224} (2010), 561--600 (electronic).

\bibitem{hr1}
A.~Hurtado and C.~Rosales, \emph{Area-stationary surfaces inside the
  sub-{R}iemannian three-sphere}, Math. Ann. \textbf{340} (2008), no.~3,
  675--708. \MR{MR2358000 (2008i:53038)}

\bibitem{hr2}
\bysame, \emph{Stability of volume-preserving area-stationary surfaces with
  non-empty singular set in sub-{R}iemannian spaces}, in preparation, 2010.

\bibitem{lopez-ros}
F.~J. L{\'o}pez and A.~Ros, \emph{Complete minimal surfaces with index one and
  stable constant mean curvature surfaces}, Comment. Math. Helv. \textbf{64}
  (1989), no.~1, 34--43. \MR{MR982560 (90b:53006)}

\bibitem{montefalcone}
F.~Montefalcone, \emph{Hypersurfaces and variational formulas in
  sub-{R}iemannian {C}arnot groups}, J. Math. Pures Appl. (9) \textbf{87}
  (2007), no.~5, 453--494. \MR{MR2322147 (2008d:53035)}

\bibitem{montgomery}
R.~Montgomery, \emph{A tour of subriemannian geometries, their geodesics and
  applications}, Mathematical Surveys and Monographs, vol.~91, American
  Mathematical Society, Providence, RI, 2002. \MR{MR1867362 (2002m:53045)}

\bibitem{mscv}
R.~Monti, F.~Serra~Cassano, and D.~Vittone, \emph{A negative answer to the
  {B}ernstein problem for intrinsic graphs in the {H}eisenberg group},
  Bollettino dell unione matematica italiana (2008), no.~3, 709--728, ISSN
  1972-6724.

\bibitem{pauls-regularity}
S.~D. Pauls, \emph{{$H$}-minimal graphs of low regularity in {$\Bbb H\sp 1$}},
  Comment. Math. Helv. \textbf{81} (2006), no.~2, 337--381. \MR{MR2225631
  (2007g:53032)}

\bibitem{posreach}
M.~Ritor\'e, \emph{Curvature measures in the {H}eisenberg group
  $\mathbb{H}^n$}, in preparation, 2010.

\bibitem{r2}
\bysame, \emph{Examples of area-minimizing surfaces in the sub-{R}iemannian
  {H}eisenberg group $\mathbb{H}^1$ with low regularity}, Calc. Var. Partial
  Differential Equations \textbf{34} (2009), no.~2, 179--192. \MR{MR2448649
  (2009h:53062)}

\bibitem{rr1}
M.~Ritor\'e and C.~Rosales, \emph{Rotationally invariant hypersurfaces with
  constant mean curvature in the {H}eisenberg group {$\Bbb H\sp n$}}, J. Geom.
  Anal. \textbf{16} (2006), no.~4, 703--720. \MR{MR2271950}

\bibitem{rr2}
\bysame, \emph{Area-stationary surfaces in the {H}eisenberg group {$\Bbb H\sp
  1$}}, Adv. Math. \textbf{219} (2008), no.~2, 633--671. \MR{MR2435652}

\bibitem{rumin}
M.~Rumin, \emph{Formes diff\'erentielles sur les vari\'et\'es de contact}, J.
  Differential Geom. \textbf{39} (1994), no.~2, 281--330. \MR{MR1267892
  (95g:58221)}

\bibitem{scott}
P.~Scott, \emph{The geometries of {$3$}-manifolds}, Bull. London Math. Soc.
  \textbf{15} (1983), no.~5, 401--487. \MR{MR705527 (84m:57009)}

\bibitem{selby}
C.~Selby, \emph{Geometry of hypersurfaces in {C}arnot groups of step 2}, Ph.D.
  thesis, Purdue University, 2006.

\bibitem{simon}
L.~Simon, \emph{Lectures on geometric measure theory}, Proceedings of the
  Centre for Mathematical Analysis, Australian National University, vol.~3,
  Australian National University Centre for Mathematical Analysis, Canberra,
  1983. \MR{MR756417 (87a:49001)}

\bibitem{strichartz}
R.~Strichartz, \emph{Sub-{R}iemannian geometry}, J. Differential Geom.
  \textbf{24} (1986), no.~2, 221--263. \MR{MR862049 (88b:53055)}

\bibitem{tanno}
S.~Tanno, \emph{Sasakian manifolds with constant {$\phi$}-holomorphic sectional
  curvature}, T\^ohoku Math. J. (2) \textbf{21} (1969), 501--507. \MR{MR0251667
  (40 \#4894)}

\bibitem{webster}
S.~M. Webster, \emph{Pseudo-{H}ermitian structures on a real hypersurface}, J.
  Differential Geom. \textbf{13} (1978), no.~1, 25--41. \MR{MR520599
  (80e:32015)}

\end{thebibliography}
\end{document}